\documentclass{siamltex}
\usepackage{amsfonts}
\usepackage{amssymb,amsmath}
\usepackage{graphicx}
\usepackage{subfigure}
\usepackage{color}
\usepackage{fullpage}

\DeclareMathAlphabet{\itbf}{OML}{cmm}{b}{it}

\def\EE{\mathbb{E}}

\def\CC{\mathbb{C}}
\def\RR{\mathbb{R}}
\def\eps{\varepsilon}
\def\bP{{\bf P}}
\def\bX{{\itbf X}}
\def\bx{{\itbf x}}
\def\bGamma{\boldsymbol{\Gamma}}

\def\om{{\omega}}

\def\cL{{\mathcal L}}

\def\cA{{\mathcal A}}
\def\cB{{\mathcal B}}
\def\cR{{\mathcal R}}
\def\cM{{\mathcal M}}

\def\cX{{\mathcal X}}
\def\cZ{{\mathcal Z}}
\renewcommand{\hat}{\widehat}
\newtheorem{remark}[theorem]{Remark}

\begin{document}

\title{Wave propagation in waveguides with random
boundaries}

\author{Ricardo Alonso\footnotemark[1],  Liliana Borcea\footnotemark[1] and Josselin
Garnier\footnotemark[2] }

\maketitle

\renewcommand{\thefootnote}{\fnsymbol{footnote}}

\footnotetext[1]{Computational    and   Applied    Mathematics,   Rice
  University,    Houston,   TX    77005.   {\tt    rja2@rice.edu   and
  borcea@rice.edu}} \footnotetext[2]{Laboratoire  de Probabilit\'es et
  Mod\`eles   Al\'eatoires   \&   Laboratoire   Jacques-Louis   Lions,
  Universit{\'e}  Paris VII,  Site Chevaleret,  75205 Paris  Cedex 13,
  France.  {\tt garnier@math.jussieu.fr}}

\renewcommand{\thefootnote}{\arabic{footnote}}

\begin{abstract}
We give a detailed analysis of long range cumulative scattering
effects from rough boundaries in waveguides. We assume small random
fluctuations of the boundaries and obtain a quantitative statistical
description of the wave field. The method of solution is based on
coordinate changes that straighten the boundaries.  The resulting
problem is similar from the mathematical point of view to that of wave
propagation in random waveguides with interior inhomogeneities. We
quantify the net effect of scattering at the random boundaries and
show how it differs from that of scattering by internal
inhomogeneities.
\end{abstract}



\section{Introduction}
\label{sect:intro}
We consider acoustic waves propagating in a waveguide with axis along
the range direction $z$. In general, the waveguide effect may be due
to boundaries or the variation of the wave speed with cross-range, as
described for example in \cite{kohler77,gomez}. We consider here only
the case of waves trapped by boundaries, and take for simplicity the
case of two dimensional waveguides with cross-section ${\mathcal D}$
given by a bounded interval of the cross-range $x$. The results extend
to three dimensional waveguides with bounded, simply connected
cross-section ${\mathcal D} \subset \mathbb{R}^2$.  

The pressure field $p(t,x,z)$ satisfies the wave equation
\begin{equation}
\label{we}
\left[\partial^2_{z} + \partial^2_x
 - \frac{1}{c^2(x)} \partial^2_t \right]p(t,x,z) = F(t,x,z) \, ,
\end{equation}
with wave speed $c(x)$ and source excitation modeled by
$F(t,x,z)$. Since the equation is linear, it suffices to consider a
point-like source located at $(x_0,z=0)$ and emitting a pulse signal
$f(t)$,
\begin{equation}
F(t,x,z) = f(t) \delta(x - x_0 ) \delta(z)  \,.
\label{eq:source}
\end{equation}
Solutions for distributed sources are easily obtained by superposing
the wave fields computed here. 

The boundaries of the waveguide are rough in the sense that they have
small variations around the values $x = 0$ and $x = X$, on a length
scale comparable to the wavelength.  Explicitly, we let
\begin{equation}
B(z) \le x \le T(z) \, , \quad \mbox{where} ~~ |B(z)| \ll X, ~ ~ 
|T(z)-X| \ll X, 
\label{eq:boundaries}
\end{equation}
and take either Dirichlet boundary conditions
\begin{equation}
p(t,x,z) = 0 \,, \quad \mbox{for} ~ x = B(z) ~~\mbox{and} ~~x = T(z),
\label{eq:Dirichlet}
\end{equation}
or mixed, Dirichlet and Neumann conditions
\begin{equation}
p(t,x= B(z),z) = 0\,, \quad \frac{\partial}{\partial n} p(t,x=T(z),z) = 0\,, 
\label{eq:mixed}
\end{equation}
where $n$ is the unit normal to the boundary $x = T(z)$.

The goal of the paper is to quantify the long range effect of
scattering at the rough boundaries.  More explicitly, to characterize
in detail the statistics of the random field $p(t,x,z)$.  This is
useful in sensor array imaging, for designing robust source or target
localization methods, as shown recently in \cite{borcea} in waveguides
with internal inhomogeneities. Examples of other applications are in
long range secure communications and time reversal in shallow water or
in tunnels \cite{garnier_papa,kuperman}.

The paper is organized as follows.  We begin in section
\ref{sect:homog} with the case of ideal waveguides, with straight
boundaries $B(z) = 0$ and $T(z) = X$, where energy propagates via
guided modes that do not interact with each other. Rough, randomly
perturbed boundaries are introduced in section \ref{sect:rand}.  The
wave speed is assumed to be known and dependent only on the
cross-range. Randomly perturbed wave speeds due to internal
inhomogeneities are considered in detail in
\cite{kohler77,kohler_wg77,dozier, garnier_papa,book07}.  Our approach
in section \ref{sect:rand} uses changes of coordinates that straighten
the randomly perturbed boundaries. We carry out the analysis in 
detail for the case of
Dirichlet boundary conditions \eqref{eq:Dirichlet} in sections \ref{sect:rand} and 
\ref{sect:diffusion}, and discuss the 
results in section \ref{sect:comparisson}. The extension to
the mixed boundary conditions \eqref{eq:mixed} is presented in section
\ref{sect:mixed}. We end in section \ref{sect:summary} with a summary.

Our approach based on changes of coordinates that straighten the boundary 
leads to a transformed problem that is similar  from
the mathematical point of view to that in waveguides with interior
inhomogeneities, so we can use the techniques from
\cite{kohler77,kohler_wg77,dozier, garnier_papa,book07} to obtain the
long range statistical characterization of the wave field in section
\ref{sect:diffusion}.  However, the cumulative scattering effects of
rough boundaries are different from those of internal inhomogeneities,
as described in section \ref{sect:comparisson}.  
We quantify these effects by estimating in a high frequency regime three important, mode dependent length scales:  
the scattering mean free path, which is the distance over which the modes lose coherence, the transport mean free path,  which is the 
distance over which the waves forget the initial direction, and the equipartition distance, over which 
the energy is uniformly distributed among the modes, independently of the initial conditions at the 
source.  We show that the random boundaries affect most strongly the high order modes,
which lose coherence rapidly, that is they have a short scattering mean free path.  Furthermore, 
these modes do not exchange efficiently energy with the other modes, so they have a longer 
transport mean free path. 
The  lower order modes can travel much longer distances before they lose their coherence 
and remarkably, their  scattering mean free path 
is similar to the transport mean  free path and to the equipartition distance. 
That is to say, in waveguides with random boundaries, when the waves travel distances that exceed 
the scattering mean free path of the low order modes, not only all the modes are incoherent, but also 
the energy is uniformly distributed among them. At such distances the wave field has lost all   information 
about the cross-range location of the source in the waveguide. 
These results can be contrasted with the situation 
with waveguides with interior random inhomogeneities, in which
the main mechanism for the loss of coherence of the 
 fields is the exchange of energy between neighboring modes \cite{kohler77,kohler_wg77,dozier, garnier_papa,book07}, so the scattering mean free paths 
 and the transport mean free paths are similar for all the modes. The low order 
 modes lose coherence much faster than in waveguides with random boundaries, and the equipartition 
 distance is longer than the scattering mean free path of these modes.  
 
\section{Ideal waveguides}
\label{sect:homog}
Ideal waveguides have straight boundaries $x = 0$ and $x = X$.  Using
separation of variables, we write the wave field as a superposition of
waveguide modes.  A waveguide mode is a monochromatic wave $P(t,x,z) =
\hat{P}(\omega,x,z) e^{- i \omega t}$ with frequency $\omega$, where
$\hat{P}(\omega,x,z)$ satisfies the Helmholtz equation
\begin{equation}
\label{eqe0}
\left[{\partial_z^2} + {\partial_x^2} + {\omega^2}/{c^{2}(x)} \right]
\hat{P}( \omega,x,z) = 0 \, , \quad z \in \mathbb{R}, ~ x \in (0,X),
\end{equation}
and either Dirichlet or mixed, Dirichlet and Neumann homogeneous
boundary conditions. The operator $\partial^2_x + \om^2/ c^{2}(x)$
with either of these conditions is self-adjoint in $L^2(0,X)$, and its
spectrum consists of an infinite number of discrete eigenvalues
$\{\lambda_j(\om)\}_{j \ge 1}$, assumed sorted in descending
order. There is a finite number $N(\omega)$ of positive eigenvalues
and an infinite number of negative eigenvalues.  The eigenfunctions
$\phi_j(\omega,x)$ are real and form an orthonormal set
\begin{equation}
\int_0^X dx \, { \phi_j(\omega,x)} \phi_l(\omega,x) = \delta_{jl}
\, , \quad j,l \ge 1,
\label{eq:orthog}
\end{equation}
where $\delta_{jl}$ is the Kronecker delta symbol.

For example, in homogeneous waveguides with $c(x) = c_o$, and for the
Dirichlet boundary conditions, the eigenfunctions and eigenvalues are
\begin{equation}
\label{eq:1}
\phi_j(x) = \sqrt{\frac{2}{X}} \mbox{sin} \left( \frac{\pi j x}{X}
\right), \qquad \lambda_j(\om) = \left(\frac{\pi}{X}\right)^2 \left[
(kX/\pi)^2 - j^2 \right], \quad \quad j = 1, 2, \ldots
\end{equation}
and the number of propagating modes is $ N(\om) = \left \lfloor
{k X}/{\pi} \right \rfloor$, where $\lfloor y \rfloor$ is the
integer part of $y$ and $k = \om/c_o$ is the homogeneous wavenumber.

To simplify the analysis, we assume that the source emits a pulse
$f(t)$ with Fourier transform
$$ \hat{f}(\omega)= \int_{-\infty}^\infty dt \, e^{i \omega t} f(t)\,, 
$$ supported in a frequency band in which the number of positive
eigenvalues is fixed, so we can set $N(\om) = N$.  We also assume that
there is no zero eigenvalue, and that the eigenvalues are simple.  The
positive eigenvalues define the modal wavenumbers
$\beta_j(\omega)=\sqrt{\lambda_j(\omega)}$ of the forward and backward
propagating modes
\[\hat{P}_j(\omega,x,z) = \phi_j(\omega,x) e^{ \pm i
\beta_j(\omega) z}, \quad j = 1, \ldots, N. \] The infinitely many
remaining modes are evanescent \[\hat{P}_j(\omega,x,z) = \phi_j(\omega
,x) e^{ - \beta_j(\omega) |z|}, \quad j > N\,, \] with wavenumber
$\beta_j (\omega)= \sqrt{-\lambda_j(\omega)}\, $.

The wave field $p(t,x,z)$ due to the source located at $(x_0,0)$ is
given by the superposition of $\hat{P}_j(\omega,x,z)$,
\begin{eqnarray*}
{p}(t,x,z) = \int \frac{d \om}{2 \pi} e^{-i \om t} \left[ \sum_{j=1}^N
\frac{\hat{a}_{j,o} (\omega)}{\sqrt{\beta_j(\omega)}} e^{i\beta_j
(\omega)z} \phi_j(\omega,x) + \sum_{j=N+1}^\infty \frac{\hat{e}_{j,o}
(\omega)}{\sqrt{\beta_j(\omega)}} e^{- \beta_j(\omega) z}
\phi_j(\omega,x) \right] {\bf 1}_{(0,\infty)}(z) + \\ \int \frac{d
\om}{2 \pi} e^{-i \om t} \left[ \sum_{j=1}^N \frac{\hat{a}_{j,o}^{\,
-} (\omega)}{\sqrt{\beta_j(\omega)}} e^{-i\beta_j (\omega) z}
\phi_j(\omega,x) + \sum_{j=N+1}^\infty \frac{\hat{e}_{j,o}^{\, -}
(\omega)}{\sqrt{\beta_j(\omega)}} e^{ \beta_j(\omega) z}
\phi_j(\omega,x) \right] {\bf 1}_{(-\infty,0)}(z) \, .
\end{eqnarray*}
The first term is supported at positive range, and it consists of
forward going modes with amplitudes $\hat a_{j,o}/\sqrt{\beta_j}$ and
evanescent modes with amplitudes $\hat e_{j,o}/\sqrt{\beta_j}$.  The
second term is supported at negative range, and it consists of
backward going and evanescent modes. The modes do not interact with
each other and their amplitudes 
\begin{eqnarray}
\hat{a}_{j,o}(\omega) &=& \hat{a}_{j,o}^{\, -}(\omega) = \frac{\hat
f(\om)}{2i \sqrt{\beta_j(\omega)}} { \phi_j(\omega,x_0 )}\,, \quad j =
1, \ldots, N, \nonumber \\ \hat{e}_{j,o}(\omega) &=&
\hat{e}_{j,o}^{\,-}(\omega) = - \frac{\hat f(\om)}{2
\sqrt{\beta_j(\omega)}} { \phi_j(\omega,x_0 )}\,, \quad j > N,
\label{eq:idealab} 
\end{eqnarray}
are determined by the source
excitation (\ref{eq:source}), which gives the jump conditions at $z=0$,
\begin{eqnarray}
&&\hat{p}(\omega, x,z=0^+)-\hat{p}(\omega, x,z=0^-)= 0 \, , \nonumber\\ &&
{\partial_z \hat{p}} (\omega,x, z=0^+) - {\partial_z \hat{p}}
(\omega,x,z=0^-) = \hat{f}(\omega) \delta( x - x_0 )\, .
\end{eqnarray}

We show next how to use the solution in the ideal waveguides as a
reference for defining the wave field in the case of randomly
perturbed boundaries.

\section{Waveguides with randomly perturbed boundaries}
\label{sect:rand}
We consider a randomly perturbed section of an ideal waveguide, over
the range interval $z\in [0,L/\eps^2]$.  There are no perturbations
for $z < 0$ and $z > L/\eps^2$. The domain of the perturbed section is
denoted by
\begin{equation}
\label{form:waveguide}
\Omega^\eps = \big\{ (x,z) \in \RR^2, ~ B(z) \leq x \leq T(z), ~ 0 < 
z <  L/\eps^2 \big\} \, ,
\end{equation}
where 
\begin{equation}
B(z) = \eps X \mu(z)\,, \quad \quad T(z) =X[1 +\eps \nu(z)]\,, \qquad
\eps \ll 1.
\label{eq:defBT}
\end{equation}
Here $\nu$ and $\mu$ are independent, zero-mean, stationary and ergodic
random processes in $z$, with covariance function
\begin{equation}
\cR_\nu(z) = \EE[ \nu(z+s) \nu(s) ] \quad \mbox{and} ~ ~ 
\cR_\mu(z) = \EE[ \mu(z+s) \mu(s) ].
\end{equation}
We assume that $\nu(z)$ and $\mu(z)$ are bounded, at least twice
differentiable with bounded derivatives, and have enough
decorrelation\footnote{Explicitly, they are $\varphi$-mixing
processes, with $\varphi \in L^{1/2}(\RR^+)$, as stated in
\cite[4.6.2]{kushner}.}. The covariance functions are normalized
so that $\cR_\nu(0)$ and $\cR_\mu(0)$ are of order one, and the magnitude
of the fluctuations is scaled by the small, dimensionless parameter
$\eps$. 

That the random fluctuations are confined to the range interval $z \in
(0,L/\eps^2)$, with $L$ an order one length scale can be motivated as
follows: By the hyperbolicity of the wave equation, we know that if we
observe $p(t,x,z)$ over a finite time window $t \in (0,T^\eps)$, the
wave field is affected only by the medium within a finite range
$L^\eps$ from the source, directly proportional to the observation
time $T^\eps$. We wish to choose $T^\eps$ large enough, in order to
capture the cumulative long range effects of scattering from the
randomly perturbed boundaries. It turns out that these effects become
significant over time scales of order $1/\eps^2$, so we take $L^\eps =
L/\eps^2$. Furthermore, we are interested in the wave field to the
right of the source, at positive range. We will see that the
backscattered field is small and can be neglected when the conditions
of the forward scattering approximation are satisfied (see Subsection \ref{secforward}).
Thus, the medium on the left of the
source has negligible influence on $p(t,x,z)$ for $z > 0$, and we may
suppose that the boundaries are unperturbed at negative range. The
analysis can be carried out when the conditions of the forward scattering
approximation are not satisfied, at considerable complication of the calculations, as
was done in \cite{garnier_solna} for waveguides with internal
inhomogeneities.

We assume here and in sections \ref{sect:diffusion} and
\ref{sect:comparisson} the Dirichlet boundary conditions
\eqref{eq:Dirichlet}. The extensions to the mixed boundary conditions
\eqref{eq:mixed} are presented in section \ref{sect:mixed}. The main
result of this section is a closed system of random differential
equations for the propagating waveguide modes, which describes the
cumulative effect of scattering of the wave field by the random
boundaries. We derive it in the following subsections and we analyze
its solution in the long range limit in section \ref{sect:diffusion}.

\subsection{Change of coordinates} We reformulate the problem in 
the randomly perturbed waveguide region $\Omega^\eps$ by changing
coordinates that straighten the boundaries,
\begin{equation}
\label{eq:6}
x = B(z) + \left[ T(z) - B(z) \right] \frac{\xi}{X}\,, \quad \xi \in
[0,X].
\end{equation}
We take this coordinate change because it is simple, but we show
later, in section \ref{sect:indc}, that the result is
independent of the choice of the change of coordinates.
In the new coordinate system, let
\begin{equation}
u(t,\xi,z) = p\left(t,B(z) + \left[ T(z) - B(z) \right] \frac{\xi}{X},
    z\right)\,, \qquad p(t,x,z) = u\left( t , \frac{(x-B(z))X}{T(z)-B(z)}
    ,z \right)\,.
\label{eq:defu}
\end{equation}
We obtain using the chain rule that the Fourier transform $\hat u(\om,\xi,z)$ 
satisfies the equation 
\begin{eqnarray}
\nonumber \partial_z^2 \hat{u} + \frac{\left[1 + \left[ (X-\xi)B' +
\xi T'\right]^2 \right]}{(T-B)^2} X^2 \partial_\xi^2 \hat{u} -
\frac{2[(X-\xi)B'+ \xi T']}{T-B}X \partial^2_{\xi z} \hat{u} + \\
\nonumber \left\{ \frac{2 B'(T'-B') }{(T-B)^2} -\frac{B''}{T-B} + \frac{\xi}{X}
\left[ 2 \left(\frac{T'-B'}{T-B}\right)^2 - \frac{T''-B''}{T-B} \right]
\right\} X \partial_\xi \hat{u} + \\ + \omega^2/
c^{2}\big(B(z)+(T(z)-B(z))\xi/X\big) \hat{u} = 0\, , 
\label{eq:pertw1}
\end{eqnarray}
for $ z \in (0,L/\eps^2)$ and $\xi \in (0,X)$. Here the prime stands
for the $z$-derivative, and the boundary conditions at $\xi = 0$ and
$X$ are
\begin{equation}
\hat{u}(\omega,0,z) =\hat{u}(\omega,X,z)=0\, .
\end{equation}
Substituting definition \eqref{eq:defBT} of $B(z)$ and $T(z)$, and
expanding the coefficients in \eqref{eq:pertw1} in series of $\eps$,
we obtain that 
\begin{equation}
\label{eq:7}
\left( \cL_0 + \eps \cL_1 + \eps^2 \cL_2 + \ldots \right) \hat
u(\om,\xi,z) = 0\,, 
\end{equation}
where 
\begin{equation}
\label{eq:8}
\cL_0 = \partial_z^2 + \partial^2_\xi + {\om^2}/{c^2(\xi)}
\end{equation}
is the unperturbed Helmholtz operator. The first and 
second order perturbation operators are given by 
\begin{equation}
\label{eq:9}
\cL_1 + \eps \cL_2 = q^\eps(\xi,z) \partial^2_{\xi z} +
\cM^\eps(\om,\xi,z)\,,
\end{equation}
with coefficient
\begin{equation}
\label{eq:10}
  q^\eps(\xi,z) = -2 \left[ (X-\xi) \mu'(z) + \xi \nu'(z) \right]
\left[1 - \eps
    \left(\nu(z)-\mu(z)\right) \right],
\end{equation}
and differential operator
\begin{eqnarray}
  \cM^\eps(\om,\xi,z) &=& - \left\{ 2 \left(\nu-\mu\right) - 3 \eps
    \left( \nu-\mu\right)^2 - \eps \left[ (X-\xi) \mu' + \xi \nu'
    \right]^2 \right\} \partial^2_\xi - \nonumber \\ && \left\{ \left[
    (X-\xi) \mu'' + \xi \nu'' \right] \left[1 -
    \eps\left(\nu-\mu\right)\right] - 2 \eps \left(\nu' -\mu'\right)
    \left[ (X-\xi) \mu' + \xi \nu' \right] \right\} \partial_\xi +
    \nonumber \\ && \om^2 \left[ (X-\xi) \mu + \xi \nu \right]
    \partial_\xi c^{-2}(\xi) + \frac{\eps \om^2}{2} \left[ (X-\xi) \mu
    + \xi \nu \right]^2 \partial^2_\xi c^{-2}(\xi)\,.
\label{eq:11}
\end{eqnarray}
The higher order terms are denoted by the dots in \eqref{eq:7}, and
are negligible as $\eps \to 0$, over the long range scale $L/\eps^2$
considered here.

\subsection{Wave decomposition and mode coupling}
\label{sect:wavedec}
Equation \eqref{eq:7} is not separable, and its solution is not a
superposition of independent waveguide modes, as was the case in ideal
waveguides. However, we have a perturbation problem, and we can use
the completeness of the set of eigenfunctions $\{ \phi_j(\om,\xi)\}_{j
\ge 1}$ in the ideal waveguide to decompose $\hat u$ in its
propagating and evanescent components,
\begin{equation}
\label{eq:WD1}
\hat u(\om,\xi,z) = \sum_{j=1}^N \phi_j(\om,\xi) \hat u_j(\om,z) + 
 \sum_{j=N+1}^\infty \phi_j(\om,\xi) \hat v_j(\om,z).
\end{equation}
The propagating components $\hat u_j$ are decomposed further in the
forward and backward going parts, with amplitudes $\hat
a_j(\om,z)$ and $\hat b_j(\om,z)$,
\begin{eqnarray}
\hat u_j = \frac{1}{\sqrt{\beta_j}} \left( \hat a_j e^{i \beta_j z} +
\hat b_j e^{-i \beta_j z} \right), \quad j = 1, \ldots, N.
\end{eqnarray}
This does not define uniquely the complex valued $\hat a_j$ and $\hat
b_j$, so we ask that they also satisfy
\begin{eqnarray}
\partial_z\hat u_j = i \sqrt{\beta_j} \left( \hat a_j e^{i \beta_j z}
- \hat b_j e^{-i \beta_j z} \right), \quad j = 1, \ldots, N.
\label{eq:WD2}
\end{eqnarray}
 This choice is motivated by the behavior of the solution in ideal
waveguides, where the amplitudes are independent of range and
completely determined by the source excitation. The expression
\eqref{eq:WD1} of the wave field is similar to that in ideal
waveguides, except that we have both forward and backward going modes,
in addition to the evanescent modes, and the amplitudes of the modes
are random functions of $z$.

The modes are coupled due to scattering at the random boundaries, as
described by the following system of random differential equations
obtained by substituting \eqref{eq:WD1} in \eqref{eq:7}, and using the
orthogonality relation \eqref{eq:orthog} of the eigenfunctions,
\begin{eqnarray}
  \partial_z \hat a_j &=& i \eps \sum_{l = 1}^N \left[ C_{jl}^\eps \,
\hat a_l e^{i (\beta_l-\beta_j)z} + \overline{C_{jl}^\eps} \, \hat b_l
e^{-i (\beta_l+\beta_j)z}\right] + \frac{i \eps}{2 \sqrt{\beta_j}}
\sum_{l = N+1}^\infty \hspace{-0.1in} e^{-i \beta_j z} \left(
Q_{jl}^\eps \, \partial_z \hat v_l + M_{jl}^\eps \, \hat v_l \right) +
O(\eps^3)\,,\quad \label{eq:WD3} \\ \partial_z \hat b_j &=& -i \eps
\sum_{l = 1}^N \left[C_{jl}^\eps \, \hat a_l e^{i (\beta_l+\beta_j)z}
+ \overline{C_{jl}^\eps} \, \hat b_l e^{-i (\beta_l-\beta_j)z}\right]-
\frac{i \eps}{2 \sqrt{\beta_j}} \sum_{l = N+1}^\infty \hspace{-0.1in}
e^{-i \beta_j z} \left( Q_{jl}^\eps \, \partial_z \hat v_l +
M_{jl}^\eps \, \hat v_l \right) + O(\eps^3)\,. \qquad \label{eq:WD4}
\end{eqnarray}
The bar denotes complex conjugation, and the coefficients are defined
below. The forward going amplitudes are determined at $z = 0$ by the
source excitation (recall \eqref{eq:idealab})
\begin{equation}
  \hat a_j(\om,0) = \hat a_{j,o}(\om)\,,\label{eq:WD5a} \quad j = 1,
  \ldots, N,
\end{equation}
and we set 
\begin{equation}
 \hat b_j \left( \om,\frac{L}{\eps^2}\right) = 0\,,\quad j = 1,
  \ldots, N,
\label{eq:WD5}
\end{equation}
because there is no incoming wave at the end of the domain.  The
equations for the amplitudes of the evanescent modes indexed by $j >
N$ are
\begin{eqnarray}
  \left( \partial^2_z - \beta_j^2\right) \hat v_j &=& - \eps \sum_{l =
    1}^N 2 \sqrt{\beta_j}\left[ C_{jl}^\eps \, \hat a_l e^{ i \beta_l
      z} + \overline{C_{jl}^\eps} \, \hat b_l e^{-i \beta_l z}
    \right]- \eps \hspace{-0.05in}\sum_{l = N+1}^\infty \left(
  Q_{jl}^\eps \, \partial_z \hat v_l + M_{jl}^\eps \, \hat v_l \right)
  + O(\eps^3)\,,
\label{eq:WD6}
\end{eqnarray}
and we complement them with the decay condition at infinity
\begin{equation}
\lim_{z \to \pm \infty} \hat v_j (\om,z) = 0 \, , \quad j >  N.
\label{eq:WD7}
\end{equation}

The coefficients
\begin{equation}
C_{jl}^\eps(\om,z) = C_{jl}^{(1)}(\om,z) + \eps C_{jl}^{(2)}(\om,z)\,,
\quad \mbox{for} ~ j \ge 1 ~ ~ \mbox{and}~ l = 1, \ldots, N,
\label{eq:WD4.q1}
\end{equation}
are defined by
\begin{eqnarray}
C_{jl}^{(1)}(\om,z) &=& \frac{1}{2 \sqrt{\beta_j(\om) \beta_l(\om)}}
\int_0^X d \xi \phi_j(\om,\xi) \cA_l(\om,\xi,z) \phi_l(\om,\xi)\,,
\label{eq:WD4C1} \\
C_{jl}^{(2)}(\om,z) &=& \frac{1}{2 \sqrt{\beta_j(\om) \beta_l(\om)}}
\int_0^X d \xi \phi_j(\om,\xi) \cB_l(\om,\xi,z) \phi_l(\om,\xi)\, , 
\label{eq:WD4C2} 
\end{eqnarray}
in terms of the linear differential operators 
\begin{eqnarray}
  \cA_l = -2 ( \nu-\mu) \partial^2_\xi - 2 i \beta_l \left[ (X-\xi)
  \mu'+\xi \nu' \right] \partial_\xi- \left[ (X-\xi) \mu''+\xi \nu''
  \right] \partial_\xi + \nonumber \\
  \om^2 \left[ (X-\xi) \mu + \xi \nu \right]
  \partial_\xi c^{-2}(\xi)\,, \qquad 
\label{eq:WD4A} 
\end{eqnarray}
and 
\begin{eqnarray}
  \cB_l = \left\{3(\nu-\mu)^2 + \left[(X-\xi) \mu' + \xi \nu'\right]^2
  \right\}\partial_\xi^2 + 2 i \beta_l (\nu-\mu)\left[(X-\xi) \mu' +
  \xi \nu'\right] \partial_\xi + \nonumber \\
  \left\{(\nu-\mu)\left[(X-\xi) \mu'' + \xi \nu''\right] + 2
  (\nu'-\mu') \left[ (X-\xi)\mu'+ \xi \nu'\right] \right\}
  \partial_\xi + \nonumber \\ \frac{\om^2}{2} \left[ (X-\xi) \mu + \xi
  \nu \right]^2 \partial^2_\xi c^{-2}(\xi)\, .
\label{eq:WD4B}
\end{eqnarray}
We also let for $j \ge 1$ and $l>N$ 
\begin{eqnarray}
Q_{jl}^\eps(\om,z) &=& \int_0^X d \xi q^\eps(\xi,z) \phi_j(\om,\xi)
\partial_\xi \phi_l(\om,\xi) = Q_{jl}^{(1)}(\om,z) + \eps
Q_{jl}^{(2)}(\om,z)\, , \nonumber \\ M_{jl}^\eps(\om,z) &=& \int_0^X d \xi
\phi_j(\om,\xi) \cM^\eps(\om,\xi,z) \phi_l(\om,\xi) =
M_{jl}^{(1)}(\om,z) + \eps M_{jl}^{(2)}(\om,z)\,.
\label{eq:QM}
\end{eqnarray}

\subsection{Analysis of the evanescent modes}
\label{sect:elim_evanesc}
We solve equations (\ref{eq:WD6}) with radiation conditions
(\ref{eq:WD7}) in order to express the amplitude of the evanescent
modes in terms of the amplitudes of the propagating modes. The
substitution of this expression in \eqref{eq:WD3}-\eqref{eq:WD4} gives
a closed system of equations for the amplitudes of the propagating
modes, as obtained in the next section.

We begin by rewriting (\ref{eq:WD6}) in
short as
\begin{equation}
  \left( \partial^2_z - \beta_j^2\right) \hat v_j +
  \eps \hspace{-0.05in} \sum_{l = N+1}^\infty \left( Q_{jl}^\eps \,
  \partial_z \hat v_l + M_{jl}^\eps \, \hat v_l \right) = -\eps
  g_j^\eps\,, \quad  \quad j >N ,
\label{eq:E3}
\end{equation}
where 
\begin{equation}
  g_j^\eps(\om,z) = g_j^{(1)}(\om,z) + \eps 
  g_j^{(2)}(\om,z) + O(\eps^3)\,, \quad  \quad j >N ,
\label{eq:E1}
\end{equation}
and
\begin{equation}
g_j^{(r)} = 2 \sqrt{\beta_j} \sum_{l = 1}^N \left[ C_{jl}^{(r)}
\, \hat a_l(\om,z) e^{ i \beta_l z} + \overline{C_{jl}^{(r)}} \, \hat
b_l e^{-i \beta_l z} \right], \quad r = 1,2 ~ ~ \mbox{and} ~  j > N.
\label{eq:E2}
\end{equation}
Using the Green's function $G_j = e^{-\beta_j |z|}/(2 \beta_j)$,
satisfying
\begin{equation}
\partial_z^2 G_j - \beta_j^2 G_j = - \delta(z)\,, \quad \lim_{|z| \to
\infty} G_j = 0\,, \quad  \quad j >N ,
\label{eq:E4}
\end{equation}
and integrating by parts, we get
\begin{equation}
\left[( {\bf I} - \eps \Psi) \hat {\itbf v}\right]_j(\om,z) = \frac{\eps}{2
\beta_j(\om)} \int_{-\infty}^\infty ds \, e^{-\beta_j(\om) |s|} 
g_j^\eps(\om,z+s)\,  , \quad  \quad j >N . \label{eq:E5}
\end{equation}
Here ${\bf I}$ is the identity and $\Psi$ is the linear integral operator
\begin{eqnarray}
[\Psi \hat {\itbf v}]_j(\om,z) &=& \frac{1}{2
\beta_j(\om)} \sum_{l= N+1}^\infty \int_{-\infty}^\infty ds \,
e^{-\beta_j(\om)|s|} \left( M_{jl}^\eps-\partial_z
Q_{jl}^\eps\right)(\om,z+s) \hat v_l(\om,z+s) + \nonumber \\ &&
\frac{1}{2} \sum_{l= N+1}^\infty \int_{-\infty}^\infty ds \,
e^{-\beta_j(\om)|s|} \mbox{sgn}(s) Q_{jl}^\eps(\om,z+s) \hat
v_l(\om,z+s)\,,
\label{eq:E6}
\end{eqnarray}
acting on the infinite vector $ \hat {\itbf v} = \left( \hat v_{N+1},
\hat v_{N+2}, \ldots \right)$ and returning an infinite vector with
entries indexed by $j$, for $j > N.$ The solvability of equation
\eqref{eq:E5} follows from the following lemma proved in appendix
\ref{sect:Proof}.
\begin{lemma}
\label{lem.1}
Let ${\mathcal L}_N$ be the space of square summable sequences of
$L^2(\mathbb{R})$ functions with linear weights, equipped with the norm 
\[
\|\hat {\itbf v} \|_{{\mathcal L}_N} = \sqrt{ \sum_{j=N+1}^\infty \left( j \|
\hat v_j\|_{L^2(\mathbb{R})}\right)^2} \, .
\]
The linear operator $\Psi: {\mathcal L}_N \to {\mathcal L}_N$ defined
component wise by \eqref{eq:E6} is bounded.
\end{lemma}

\vspace{0.1in} \noindent 
Thus, the inverse operator is 
\[
(I-\eps \Psi)^{-1} = I + \eps \Psi + \ldots,
\]
and the solution of (\ref{eq:E5}) is given by
\begin{equation}
\hat v_j(\om,z) = \frac{\eps}{2 \beta_j(\om)} \int_{-\infty}^\infty ds
\, e^{-\beta_j(\om)|s|} g_j^{(1)}(\om,z+s) + O(\eps^2)\, .
\label{eq:E7}
\end{equation}
Using definition (\ref{eq:E2}) and the fact that the $z$ derivatives
of $\hat a_l$ and $\hat b_l$ are of order $\eps$, we get
\begin{eqnarray}
\hat v_j(\om,z) &=& \frac{\eps}{\sqrt{\beta_j(\om)}} \sum_{l=1}^N \hat
a_l(\om,z) e^{i \beta_l z} \int_{-\infty}^\infty ds \,
e^{-\beta_j(\om)|s|+ i \beta_l(\om) s} C_{jl}^{(1)}(\om,z+s) +
\nonumber \\ && \frac{\eps}{\sqrt{\beta_j(\om)}} \sum_{l=1}^N \hat
b_l(\om,z) e^{-i \beta_l z}\int_{-\infty}^\infty ds \,
e^{-\beta_j(\om)|s|- i \beta_l(\om) s}
\overline{C_{jl}^{(1)}(\om,z+s)} + O(\eps^2)\,.
\label{eq:E8}
\end{eqnarray}

We also need
\begin{equation}
\hat w_j(\om,z) = \partial_z \hat v_j(\om,z)\,,
\label{eq:E9}
\end{equation}
which we compute by taking a $z$ derivative in (\ref{eq:E3}) and using
the radiation condition $\hat w_j(\om, z) \to 0$ as $|z| \to
\infty$. The resulting equation is similar to (\ref{eq:E5})
\begin{eqnarray} 
\left[ ({\bf I} - \eps \tilde \Psi) {\itbf w}
\right]_j\hspace{-0.05in}(\om,z) &=& \frac{\eps}{2}
\int_{-\infty}^\infty \hspace{-0.1in} ds \, e^{-\beta_j(\om) |s|}
 \left[ \mbox{sgn}(s) g_j^\eps(\om,z+s) +
 \hspace{-0.05in}
 \sum_{l=N+1}^\infty
  \hspace{-0.05in}
  M_{jl}^\eps(\om,z+s)
 \hat v_l(\om,z+s)\right], \qquad
\label{eq:E10}
\end{eqnarray}
where we integrated by parts and introduced the linear integral
operator
\begin{eqnarray}
[\tilde \Psi \hat {\itbf w}]_j(\om,z) &=& \frac{1}{2}
\sum_{l= N+1}^\infty \int_{-\infty}^\infty ds \, e^{-\beta_j(\om)|s|}
\mbox{sgn}(s) Q_{jl}^\eps(\om,z+s) \hat w_l(\om,z+s)\,.
\label{eq:E12}
\end{eqnarray}
This operator is very similar to $\Psi$ and it is bounded, as follows
from the proof in appendix \ref{sect:Proof}. Moreover, substituting
expression (\ref{eq:E8}) of $\hat v_l$ in (\ref{eq:E10}) we obtain
after a calculation that is similar to that in appendix
\ref{sect:Proof} that the series in the index $l$ is
convergent. Therefore, the solution of (\ref{eq:E10}) is
\begin{equation} 
\hat w_j(\om,z) = \frac{\eps}{2} \int_{-\infty}^\infty ds \,
e^{-\beta_j(\om) |s|} \mbox{sgn}(s) g_j^\eps(\om,z+s) + O(\eps^2)
\label{eq:E11}
\end{equation}
and more explicitly, 
\begin{eqnarray}
\partial_z \hat v_j(\om,z) &=& \eps \sqrt{\beta_j(\om)} \sum_{l=1}^N
\hat a_l(\om,z) e^{i \beta_l z} \int_{-\infty}^\infty ds \,
e^{-\beta_j(\om)|s|+ i \beta_l(\om) s} \mbox{sgn}(s)
C_{jl}^{(1)}(\om,z+s) + \nonumber \\ && {\eps}{\sqrt{\beta_j(\om)}}
\sum_{l=1}^N \hat b_l(\om,z) e^{-i \beta_l z}\int_{-\infty}^\infty ds
\, e^{-\beta_j(\om)|s|- i \beta_l(\om) s} \mbox{sgn}(s)
\overline{C_{jl}^{(1)}(\om,z+s)} + O(\eps^2)\,.
\label{eq:E14}
\end{eqnarray}

\subsection{The closed system of equations for the propagating modes}
\label{sect:Closed}
The substitution of equations (\ref{eq:E8}) and (\ref{eq:E14}) in
(\ref{eq:WD3}) and (\ref{eq:WD4}) gives the main result of this
section: a closed system of differential equations for the propagating
mode amplitudes. We write it in compact form using the $2N$ vector
\begin{equation} 
{\bX}_\om(z) = \left[ \begin{array}{c} \hat {\itbf a}(\om,z) \\ \hat
    {\itbf b}(\om,z)
\end{array} \right]\, , 
\label{eq:C1}
\end{equation}
obtained by concatenating vectors $\hat {\itbf a}(\om,z)$ and 
$\hat {\itbf b}(\om,z)$ with components $\hat a_j(\om,z)$ and $\hat b_j(\om,z)$,
for $j = 1, \ldots, N$. We have
\begin{equation}
\partial_z {\bX}_\om(z) = \eps {\bf H}_\om(z){\bX}_\om(z) +  \eps^2 
{\bf G}_\om(z){\bX}_\om(z) + O(\eps^3)\,,
\label{eq:C2}
\end{equation}
with $2N \times 2 N$ complex matrices given in block form by 
\begin{equation}
\label{eq:defH}
{\bf H}_\omega(z) =
\left[ \begin{array}{cc}
{\bf H}^{(a)}_\omega(z) & {\bf H}^{(b)}_\omega(z) \\
\overline{{\bf H}^{(b)}_\omega}(z) & \overline{{\bf H}^{(a)}_\omega}(z)\\
\end{array} \right]  \, , \ \ \ \ \
{\bf G}_\omega(z) =
\left[ \begin{array}{cc}
{\bf G}^{(a)}_\omega(z) & {\bf G}^{(b)}_\omega(z) \\
\overline{{\bf G}^{(b)}_\omega}(z) & \overline{{\bf G}^{(a)}_\omega}(z)\\
\end{array} \right].
\end{equation}
The entries of the blocks in ${\bf H}_\om$ are 
\begin{eqnarray}
\label{defHja}
H^{(a)}_{\omega,jl} (z) = i C_{jl}^{(1)}(\om,z) e^{ i
(\beta_l-\beta_j) z} \, ,\ \ \ \ \ \ H^{(b)}_{\omega,jl} (z) = i
C_{jl}^{(1)}(\om,z) e^{-i(\beta_l+\beta_j)z} \, ,
\label{eq:Hja}
\end{eqnarray}
and the entries of the blocks in ${\bf G}_\om$ are
\begin{eqnarray}
G^{(a)}_{\omega,jl} (z) &=& i e^{ i (\beta_l-\beta_j) z}
C_{jl}^{(2)}(\om,z) + i e^{ i (\beta_l-\beta_j) z} \hspace{-0.1in}
\sum_{l'=N+1}^\infty \frac{M_{jl'}^{(1)}(\om,z)}{ 2 \sqrt{\beta_j
\beta_{l'}}} \int_{-\infty}^\infty ds \, e^{-\beta_{l'} |s| + i
\beta_l s} C_{l'l}^{(1)}(\om,z+s) + \nonumber \\ && i e^{ i
(\beta_l-\beta_j) z}\hspace{-0.1in}\sum_{l'=N+1}^\infty
\frac{Q_{jl'}^{(1)}(\om,z)}{2 \sqrt{\beta_j \beta_{l'}}}
\int_{-\infty}^\infty ds \, e^{-\beta_{l'} |s| + i \beta_l s}
\beta_{l'} \, \mbox{sgn}(s) \, C_{l'l}^{(1)}(\om,z+s)\,, \label{eq:Gja}
\\ G^{(b)}_{\omega,jl} (z) &=& i e^{ -i (\beta_l+\beta_j) z}
C_{jl}^{(2)}(\om,z) - i e^{ -i (\beta_l+\beta_j)
z}\hspace{-0.1in}\sum_{l'=N+1}^\infty \frac{M_{jl'}^{(1)}(\om,z)}{ 2
\sqrt{\beta_j \beta_{l'}}} \int_{-\infty}^\infty ds \, e^{-\beta_{l'}
|s| - i \beta_l s} \overline{C_{l'l}^{(1)}(\om,z+s)} + \nonumber \\ &&
i e^{ -i (\beta_l+\beta_j) z}\hspace{-0.1in}\sum_{l'=N+1}^\infty
\frac{Q_{jl'}^{(1)}(\om,z)}{2 \sqrt{\beta_j \beta_{l'}}}
\int_{-\infty}^\infty ds \, e^{-\beta_{l'} |s| - i \beta_l s}
\beta_{l'}\, \mbox{sgn}(s) \, \overline{C_{l'l}^{(1)}(\om,z+s)} \, .
\label{eq:Gjb}
\end{eqnarray}
The coefficients in (\ref{eq:Hja})-(\ref{eq:Gjb}) are defined in terms
of the random functions $\nu(z)$, $\mu(z)$, their derivatives, and the
following integrals,
\begin{eqnarray} 
c_{\nu,jl}(\om) &=& \frac{1}{2 \sqrt{\beta_j \beta_l}} \int_0^X d \xi \, 
  \phi_j(\xi) \left[ - 2 \partial_\xi^2 + \om^2 \xi \partial_\xi
    c^{-2}(\xi)\right] \phi_l(\xi)\,,
\label{eq:cnu} \\
c_{\mu,jl}(\om) &=& \frac{1}{2 \sqrt{\beta_j \beta_l}} \int_0^X d \xi \, 
  \phi_j(\xi) \left[ 2 \partial_\xi^2 + \om^2 (X-\xi) \partial_\xi
    c^{-2}(\xi)\right] \phi_l(\xi)\, ,
\label{eq:cmu} \\
d_{\nu,jl}(\om) &=& -\frac{1}{2 \sqrt{\beta_j \beta_l}} \int_0^X d \xi \, 
  \xi \,  \phi_j(\xi)  \partial_\xi \phi_l(\xi) \, ,
\label{eq:dnu} \\
d_{\mu,jl}(\om) &=& -\frac{1}{2 \sqrt{\beta_j \beta_l}} \int_0^X d \xi \, 
(X-\xi)\, \phi_j(\xi) \partial_\xi \phi_l(\xi)\, ,
\label{eq:dmu} 
\end{eqnarray}
satisfying the symmetry relations 
\begin{eqnarray}
c_{\nu,jl}(\om) &=& c_{\nu,lj}(\om)\,, \nonumber \\ c_{\mu,jl}(\om) &=&
c_{\mu,lj}(\om)\,, \nonumber \\ d_{\nu,jl}(\om) + d_{\nu,lj}(\om) &=&
\frac{\delta_{jl}}{2 \sqrt{\beta_j(\om) \beta_l(\om)}}\,,\nonumber \\
d_{\mu,jl}(\om) + d_{\mu,lj}(\om) &=& -\frac{\delta_{jl}}{2
\sqrt{\beta_j(\om) \beta_l(\om)}}\,.
\label{eq:symmetries}
\end{eqnarray}
We have from (\ref{eq:WD4C1}) that
\begin{eqnarray}
C_{jl}^{(1)}(\om,z) = \nu(z) c_{\nu,jl}(\om) + \left[ \nu''(z) + 2 i
  \beta_l(\om) \nu'(z) \right] d_{\nu,jl}(\om) + \nonumber \\\mu(z)
c_{\mu,jl}(\om) + \left[ \mu''(z) + 2 i \beta_l(\om) \mu'(z) \right]
d_{\mu,jl}(\om)\,,
\label{eq:Cjl1}
\end{eqnarray}
and from (\ref{eq:QM}), (\ref{eq:10}), (\ref{eq:11}) that 
\begin{eqnarray}
\frac{Q_{jl'}^{(1)}(\om,z)}{2\sqrt{\beta_j(\om) \beta_{l'}(\om)}} &=&
2 \left[ \nu'(z) d_{\nu,jl'}(\om) + \mu'(z) d_{\mu,jl'}(\om)\right], \nonumber
\\ \frac{M_{jl'}^{(1)}(\om,z)}{2\sqrt{\beta_j(\om) \beta_{l'}(\om)}}
&=& \nu(z) c_{\nu,jl'}(\om) + \mu(z) c_{\mu,jl'}(\om)+ \nu''(z) d_{\nu,jl'}(\om) +
\mu''(z)d_{\mu,jl'}(\om)\,.
\end{eqnarray}
\section{The long range limit}
\label{sect:diffusion}
In this section we use the system \eqref{eq:C2} to quantify the
cumulative scattering effects at the random boundaries. We begin with
the long range scaling chosen so that these effects are
significant. Then, we explain why the backward going amplitudes are
small and can be neglected. This is the forward scattering
approximation, which gives a closed system of random differential
equations for the amplitudes $\{\hat a_j\}_{j = 1, \ldots, N}$. We use
this system to derive the main result of the section, which says that
the amplitudes $\{\hat a_j\}_{j = 1,\ldots, N}$ converge in
distribution as $\eps \to 0$ to a diffusion Markov process, whose
generator we compute explicitly. This allows us to calculate all the
statistical moments of the wave field.

\subsection{Long range scaling}
\label{sec:chap21_propmatr}%
It is clear from \eqref{eq:C1} that since the right hand side is
small, of order $\eps$, there is no net effect of scattering from the
boundaries over ranges of order one. If we considered ranges of order
$1/\eps$, the resulting equation would have an order one right hand
side given by ${\bf H}_\om(z/\eps) {\bX}_\om(z/\eps)$, but this
becomes negligible as well for $\eps \to 0$, because the expectation
of ${\bf H}_\om(z/\eps)$ is zero \cite[Chapter 6]{book07}. We need
longer ranges, of order $1/\eps^2$ to see the effect of scattering
from the randomly perturbed boundaries.

Let then $\hat{a}_j^\eps$, $\hat{b}_j^\eps$ be the rescaled amplitudes
\begin{equation}
\hat{a}_j^\eps(\omega,z) = \hat{a}_j \left( \omega,\frac{z}{ \eps^{2}} \right)  ,
\hspace{0.3in} \hat{b}_j^\eps(\omega,z) = \hat{b}_j \left(\omega,
\frac{z}{ \eps^{2}} \right) ,   \quad j = 1, \ldots, N,
\end{equation}
and obtain from \eqref{eq:C2} that $\bX^\eps_\omega(z) =
\bX_\om(z/\eps^2)$ satisfies the equation
\begin{equation}
\label{eq:PP1}
\frac{d\bX^\eps_\omega(z)}{dz}= \frac{1}{\eps} {\bf
H}_\omega\left(\frac{z}{\eps^2} \right) \bX^\eps_\omega(z) + {\bf
G}_\omega\left(\frac{z}{\eps^2} \right) \bX^\eps_\omega(z)\,, \quad 0 < z < L,
\end{equation}
with boundary conditions 
\begin{equation}
\label{eq:PP2}
\hat a_j^\eps (\om,0) = \hat a_{j,o}, \quad 
\hat b_j^\eps (\om,L) = 0, \quad j = 1, \ldots, N. 
\end{equation}
We can solve it using the complex valued, random propagator matrix
$\bP^\eps_\omega(z) \in \mathbb{C}^{2N \times 2N}$, the solution of
the initial value problem
\begin{equation}
\frac{d \bP^\eps_\omega(z)}{dz}= \frac{1}{\eps} {\bf H}_\omega
\left(\frac{z}{\eps^2} \right) \bP^\eps_\omega(z) + {\bf G}_\omega
\left(\frac{z}{\eps^2} \right) \bP^\eps_\omega(z) \, \quad \mbox{for }
z > 0, ~ ~ \mbox{and } \bP^\eps_\omega(0) = {\bf I}.
\label{eq:IVPP}
\end{equation}
The solution is
\[
\bX^\eps_\omega(z) = \bP^\eps_\omega(z) \left[ \begin{array}{c}
\hat{\itbf a}_0(\omega) \\ \hat{\itbf b}^\eps(\omega,0)
\end{array} \right],
\]
and $\hat{\itbf b}^\eps(\omega,0)$ can be eliminated from the boundary
identity
\begin{equation}
\left[ \begin{array}{c}
\hat{\itbf a}^\eps (\omega,L)\\ 
{\bf 0}
\end{array} \right]
=
\bP^\eps_\omega(L)
\left[ \begin{array}{c}
\hat{\itbf a}_0(\omega) \\ \hat{\itbf b}^\eps(\omega,0)
\end{array} \right] \, .
\end{equation}
Furthermore, it follows from the symmetry relations (\ref{eq:defH})
satisfied by the matrices ${\bf H}_\omega$ and ${\bf G}_\omega$ that
the propagator has the block form
\begin{equation}
\label{formpropagator}
\bP^\eps_\omega(z) = \left[ \begin{array}{cc}
{\bf P}^{\eps,a}_\omega(z) & {{\bf P}^{\eps,b}_\omega(z)}\\
\overline{{\bf P}^{\eps,b}_\omega(z)} & \overline{{\bf P}^{\eps,a}_\omega(z)}\\
\end{array} \right] \, ,
\end{equation}
where ${\bf P}^{\eps,a}_\omega(z)$ and ${\bf P}^{\eps,b}_\omega(z)$
are $N \times N$ complex matrices.  The first block ${\bf
P}^{\eps,a}_\omega$ describes the coupling between different
forward going modes, while ${\bf P}^{\eps,b}_\omega$ describes the
coupling between forward going and backward going modes.

\subsection{The diffusion approximation}
\label{sect:diffThm}%
The limit $\bP^\eps_\omega$ as $\eps \rightarrow 0$ can be obtained
and identified as a multi-dimensional diffusion process, meaning that
the entries of the limit matrix satisfy a system of linear stochastic
equations.  This follows from the application of the
diffusion approximation theorem proved in \cite{kohler74}, which
applies to systems of the general form
\begin{equation}
\label{eq:DIF1}
\frac{d {\mathbf{\mathcal X}}^\eps(z)}{dz} = \frac{1}{\eps}
     {\mathbf{\mathcal F}}\left({\mathbf{\mathcal
         X}}^\eps(z),{\mathbf{\mathcal Y}}\left(\frac{z}{\eps^2}
     \right), \frac{z}{\eps^2} \right) + {\mathbf{\mathcal
         G}}\left({\mathbf{\mathcal X}}^\eps(z),{\mathbf{\mathcal
         Y}}\left(\frac{z}{\eps^2} \right), \frac{z}{\eps^2} \right)
      \quad \mbox{for} ~ z > 0,   \quad \mbox{and} ~ ~{\mathbf{\mathcal X}}^\eps(0) =
        {\mathbf{\mathcal X}}_o,
\end{equation}
for a vector or matrix ${\mathbf{\mathcal X}}^\eps(z)$ with real
entries. The system is driven by a stationary, mean zero and mixing
random process $ {\mathbf{\mathcal Y}}(z)$.  The functions
${\mathbf{\mathcal F}}(\chi,y,\tau)$ and ${\mathbf{\mathcal
    G}}(\chi,y,\tau)$ are assumed at most linearly growing and smooth
in $\chi$, and the dependence in $\tau$ is periodic or almost periodic
\cite[Section 6.5]{book07}. The function ${\mathbf{\mathcal
    F}}(\chi,y,\tau)$ must also be centered: For any fixed $\chi$ and
$\tau$, $\EE[{\mathbf{\mathcal F}}(\chi,{\mathbf{\mathcal Y}}(0),\tau)] =
0$.

The diffusion approximation theorem states that as $\eps \to 0$,
${\mathbf{\mathcal X}}^\eps(z)$ converges in distribution to the
diffusion Markov process ${\mathbf{\mathcal X}}(z)$ with generator
$\mathcal L$, acting on sufficiently smooth functions $\varphi(\chi)$
as
\begin{eqnarray}
\mathcal L \varphi(\chi) = \lim_{T\to \infty} \frac{1}{T} \int_0^T d
\tau \int_0^\infty dz \, \EE \left[ {\mathbf{\mathcal
      F}}(\chi,{\mathbf{\mathcal Y}}(0),\tau) \cdot \nabla_\chi
  \left[ {\mathbf{\mathcal F}}(\chi,{\mathbf{\mathcal
        Y}}(z),\tau) \cdot \nabla_\chi \varphi(\chi) \right] \right] + \nonumber \\
\frac{1}{T} \int_0^T d \tau \, \EE \left[ {\mathbf{\mathcal
      G}}(\chi,{\mathbf{\mathcal Y}}(0),\tau) \cdot \nabla_\chi
  \varphi(\chi) \right] \, .
\end{eqnarray}
To apply it to the initial value problem \eqref{eq:IVPP} for the
complex $2N \times 2N $ matrix ${\bf P}_\om^\eps(z)$, we let
${\mathbf{\mathcal X}}^\eps(z)$ be the matrix obtained by
concatenating the absolute values and phases of the entries in ${\bf
P}_\om^\eps(z)$.  The driving random process ${\mathbf{\mathcal Y}}$
is given by $\mu(z), \nu(z) $ and their derivatives, which are
stationary, mean zero and mixing by assumption. The expression of 
functions ${\mathbf{\mathcal F}}$ and ${\mathbf{\mathcal G}}$ 
follows from \eqref{eq:IVPP} and the chain rule.   The dependence on
the fast variable $\tau = z/\eps^2$ is in the arguments of $\cos$ and
$\sin$ functions, the real and imaginary parts of the complex
exponentials in \eqref{eq:Hja}-\eqref{eq:Gjb}.

\subsection{The forward scattering approximation}
\label{secforward}%
When we use the diffusion-approximation theorem in \cite{kohler74}, we
obtain that the limit entries of ${\bf P}^{\eps,b}_\omega(z)$ are
coupled to the limit entries of ${\bf P}^{\eps,a}_\omega(z)$ through
the coefficients 
$$ \hat \cR_\nu(\beta_j + \beta_l) = 2 \int_{0}^\infty dz \, \cR_\nu(z)
\cos[( \beta_j+\beta_l )z]\, , \ \ \ \
\hat \cR_\mu(\beta_j + \beta_l) = 2 \int_{0}^\infty dz \, \cR_\mu(z)
\cos[( \beta_j+\beta_l )z]\, ,
$$ for $j,l=1,\ldots, N$.  Here $\hat \cR_\nu$ and $\hat \cR_\mu$ are
the power spectral densities of the processes $\nu$ and $\mu$, the
Fourier transform of their covariance functions. They are
evaluated at the sum of the wavenumbers $\beta_j + \beta_l$ because
the phase factors present in the matrix ${\bf H}^{(b)}_\omega(z)$ are
$\pm(\beta_j+\beta_l)z$. The limit entries of ${\bf
P}^{\eps,a}_\omega(z)$ are coupled to each other through the power
spectral densities evaluated at the difference of the wavenumbers,
$\hat \cR_\nu(\beta_j - \beta_l)$ and $\hat \cR_\mu(\beta_j -
\beta_l)$, for $j,l=1,\ldots, N$, because the phase factors in the
matrix ${\bf H}^{(a)}_\omega(z)$ are $\pm(\beta_j-\beta_l)z$. Thus, if
we assume that the power spectral densities are small at large
frequencies, we may make the approximation
\begin{equation}
\label{validforward}
\hat \cR_\nu(\beta_j + \beta_l) \approx 0\,, \qquad \hat
\cR_\mu(\beta_j + \beta_l) \approx 0 \,, \quad \mbox{for} ~ ~j,l = 1,
\ldots, N,
\end{equation}
which implies that we can neglect coupling between the forward and backward
propagating modes as $\eps \to 0$. The forward going modes remain
coupled to each other, because at least some combinations of the
indexes $j,l$, for instance those with $|j-l|=1$, give non-zero
coupling coefficients $\hat \cR_\nu(\beta_j-\beta_l)$ and $\hat
\cR_\mu( \beta_j-\beta_l)$. 

Because the backward going mode amplitudes satisfy the homogeneous end condition
$\hat b_{j}^\eps (\om,L) = 0$, and because they are asymptotically
uncoupled from $\{\hat a_j^\eps\}_{j = 1, \ldots, N}$, we can set them
to zero. This is the forward scattering approximation, where the
forward propagating mode amplitudes satisfy the closed system
\begin{equation}
\label{evola}
\frac{d \hat{\itbf a}^\eps}{dz} = \frac{1}{\eps} {\bf H}^{(a)}_\omega
\left(\frac{z}{\eps^2} \right) \hat{\itbf a}^\eps + {\bf
  G}^{(a)}_\omega \left(\frac{z}{\eps^2} \right) \hat{\itbf a}^\eps \,
\quad \mbox{for}  ~ z > 0, ~ ~ \mbox{and} ~ \hat{a}^\eps_j(\omega,z=0)=
{\hat{a}_{j,o}}(\omega).
\end{equation}

\begin{remark}
\label{rem.1}
Note that the matrix ${\bf H}^{(a)}_\omega$ is not skew Hermitian,
which implies that for a given $\eps$ there is no conservation of
energy of the forward propagating modes, over the randomly perturbed
region,
\[
\sum_{j=1}^N | \hat{a}_j^\eps (L) |^2 \ne \sum_{j=1}^N | \hat{a}_{j,o}
|^2.
\]
This is due to the local exchange of energy between the propagating
and evanescent modes.  However, we will see that the energy of the
forward propagating modes is conserved in the limit $\eps \to 0$.
\end{remark}

\subsection{The coupled mode diffusion process}
\label{subseccoupledpower}
We now apply the diffusion approximation theorem to the system
(\ref{evola}) and obtain after a long calculation that we do not
include for brevity, the main result of this section:

\begin{theorem}
\label{propdiff}%
The complex mode amplitudes $\{\hat{a}_j^\eps(\omega,z)
\}_{j=1,\ldots,N}$ converge in distribution as $\eps \rightarrow 0$ to
a diffusion Markov process process $\{\hat{a}_j(\omega,z)
\}_{j=1,\ldots,N}$ with generator ${\cal L}$ given below. 
\end{theorem}

\vspace{0.1in}
\noindent 
Let us write the limit process as
$$ \hat{a}_j(\omega,z) = P_j(\omega,z)^{1/2} e^{i \theta_j(\omega,z)},
\quad j=1,\ldots,N,
$$ in terms of the power $|\hat a_j|^2 = P_j$ and the phase
$\theta_j$. Then, we can express the infinitesimal generator ${\cal
  L}$ of the limit diffusion as the sum of two operators
\begin{eqnarray}
\label{gendiffa}
{\cal L} &=& {\cal L}_P + {\cal L}_\theta . 
\end{eqnarray}
The first is a partial differential operator in the powers
\begin{eqnarray}
\label{gendiff2P}
{\cal L}_P =\sum_{{\scriptsize \begin{array}{c}j, l  = 1 \\
j \ne l \end{array}} }^N
\Gamma_{jl}^{(c)}(\omega) \left[ P_l P_j \left(
  \frac{\partial}{\partial P_j} -\frac{\partial}{\partial P_l} \right)
  \frac{\partial}{\partial P_j} + (P_l-P_j) \frac{\partial}{\partial
    P_j} \right] \, ,
\end{eqnarray}
with matrix $\bGamma^{(c)}(\om)$ of coefficients that are non-negative
off the diagonal, and sum to zero in the rows
\begin{equation}
\label{defgamma1b}
\Gamma_{jj}^{(c)}(\omega) = - \sum_{l \neq j}
\Gamma_{jl}^{(c)}(\omega)\, .
\end{equation}
The off-diagonal entries are defined by the power spectral densities
of the fluctuations $\nu$ and $\mu$, and the derivatives of the
eigenfunctions at the boundaries,
\begin{eqnarray}
 \Gamma_{jl}^{(c)}(\omega) = \frac{X^2}{4 \beta_j(\om) \beta_l(\om)}
 \left\{ \left[\partial_\xi \phi_j(\om,X) \partial_\xi
   \phi_l(\om,X)\right]^2 \hat \cR_\nu[\beta_j(\om)-\beta_l(\om)] +
 \right. \nonumber \\ \left.  \left[\partial_\xi \phi_j(\om,0)
   \partial_\xi \phi_l(\om,0)\right]^2 \hat
 \cR_\mu[\beta_j(\om)-\beta_l(\om)]\right\} \, \label{defgamma1} .
\end{eqnarray} 
The second partial differential operator is with respect to the phases
\begin{eqnarray}
\nonumber 
{\cal L}_\theta = \frac{1}{4} \sum_{{\scriptsize \begin{array}{c}j, l  = 1 \\
j \ne l \end{array}}}^N
\Gamma_{jl}^{(c)}(\omega) \left[ \frac{P_j}{P_l} \frac{\partial^2}{
    \partial \theta_l^2} + \frac{P_l}{P_j} \frac{\partial^2}{ \partial
    \theta_j^2} + 2 \frac{\partial^2}{\partial \theta_j \partial
    \theta_l} \right] + \frac{1}{2} \sum_{j , l=1}^N
\Gamma_{jl}^{(0)}(\omega) \frac{\partial^2}{\partial \theta_j \partial
  \theta_l} + \\ \frac{1}{2} \sum_{{\scriptsize \begin{array}{c}j, l  = 1 \\
j \ne l \end{array}} }^N
\Gamma_{jl}^{(s)}(\omega) \frac{\partial}{\partial \theta_j} + \sum_{j=1}^N
\kappa_j(\om) \frac{\partial}{\partial \theta_j} \,, \label{gendiff2T}
\end{eqnarray}
with nonnegative coefficients
\begin{eqnarray}
\Gamma_{jl}^{(0)}(\omega) = \frac{X^2}{4 \beta_j(\om) \beta_l(\om)}
 \left\{ \left[\partial_\xi \phi_j(\om,X) \partial_\xi
   \phi_l(\om,X)\right]^2 \hat \cR_\nu(0) +
 \right. \nonumber \\ \left.  \left[\partial_\xi \phi_j(\om,0)
   \partial_\xi \phi_l(\om,0)\right]^2 \hat
 \cR_\mu(0)\right\} \, \label{defgamma10} ,
\end{eqnarray} 
and 
\begin{eqnarray}
 \Gamma_{jl}^{(s)}(\omega) = \frac{X^2}{4 \beta_j(\om) \beta_l(\om)}
 \left\{ \left[\partial_\xi \phi_j(\om,X) \partial_\xi
   \phi_l(\om,X)\right]^2 \gamma_{\nu,jl}(\om) + \right. \nonumber
 \\ \left.  \left[\partial_\xi \phi_j(\om,0) \partial_\xi
   \phi_l(\om,0)\right]^2 \gamma_{\mu,jl}(\om)\right\}
 \, \label{defgamma1s} , 
\end{eqnarray} 
for $j \ne l$, where
\begin{eqnarray}
\label{eq:gammanumu}
\gamma_{\nu,jl} (\omega)&=& 2\int_{0}^\infty dz \, \sin \left[
  (\beta_j(\omega)-\beta_l(\omega))z \right] \cR_\nu(z) \,
,\\ \gamma_{\mu,jl} (\omega)&=& 2\int_{0}^\infty dz \, \sin \left[
  (\beta_j(\omega)-\beta_l(\omega))z \right] \cR_\mu(z) \, .
\end{eqnarray}
The diagonal part of $\Gamma^{(s)}(\om)$ is defined by  
\begin{equation}
\label{eq:defgamma1s}
\Gamma_{jj}^{(s)}(\om) = - \sum_{\l \ne j} \Gamma_{jl}^{(s)}(\om).
\end{equation} 
All the terms in the generator except for the last one in
\eqref{gendiff2T} are due to the direct coupling of the propagating
modes. The coefficient $\kappa_j$ in the last term is
\begin{equation}
\kappa_j(\om) = \kappa_j^{(a)}(\om) + \kappa_j^{(e)}(\om),
\label{eq:kappa1}
\end{equation}
with the first part due to the direct coupling of the propagating
modes and given by
\begin{eqnarray}
\nonumber \kappa_j^{(a)} &=& \cR_\nu(0) \left\{
\int_0^X \hspace{-0.03in} d \xi \left[ \frac{\om^2}{4 \beta_j} \xi^2
  \phi_j^2 \, \partial_\xi^2 c^{-2} - \frac{3}{2 \beta_j}
  (\partial_\xi \phi_j)^2 \right] +
 \hspace{-0.03in} \sum_{l \ne j, l = 1}^N (\beta_l + \beta_j) \big[
   d_{\nu,jl}^2 (\beta_l^2-\beta_j^2) +2 d_{\nu,jl}c_{\nu,jl}\big]
 \right\} - \nonumber \\ && \cR_\nu''(0) \left\{ \frac{1}{4 \beta_j} -
 \frac{1}{2 \beta_j} \int_0^X d \xi \, \xi^2 ( \partial_\xi \phi_j )^2
 +
\hspace{-0.03in} \sum_{l \ne j, l = 1}^N (\beta_l-\beta_j) d_{\nu,jl}^2 
\right\} ~ ~ + ~ ~ \mu \mbox{ terms, }
\label{eq:kappaa}
\end{eqnarray}
with the abbreviation ``$\mu$ terms'' for the similar contribution of
the $\mu$ process.  The coupling via the evanescent modes determines
the second term in (\ref{eq:kappa1}), and it is given by
\begin{eqnarray}
\kappa_j^{(e)} &=& \sum_{l=N+1}^\infty \frac{ X^2 \left[ \partial_\xi
    \phi_j(X) \partial_\xi \phi_l(X)\right]^2}{2 \beta_j \beta_l
  (\beta_j^2 + \beta_l^2)^2} \int_0^\infty ds \, e^{-\beta_l s}
\cR_\nu''(s) \left[ (\beta_l^2 - \beta_j^2)
  \cos(\beta_j s) - 2 \beta_j \beta_l \sin(\beta_j s) \right] +
\nonumber \\ && \sum_{l=N+1}^\infty 2 \beta_l \left[-d_{\nu,lj}^2
  \cR_\nu''(0) + \frac{c_{\nu,lj}^2}{\beta_j^2 +
    \beta_l^2} \cR_\nu(0) \right] ~ ~ + ~ ~ \mu \mbox{ terms.}
\label{eq:kappa_e}
\end{eqnarray}

\subsubsection{Discussion}
\label{sect:discuss}
We now describe some properties of the diffusion process $\hat{\itbf
  a}$: 
\vspace{0.1in}
\begin{enumerate}
\itemsep 0.1in 
\item 
Note that the coefficients of the partial derivatives in $P_j$ of the
infinitesimal generator ${\cal L}$ depend only on $\{P_l\}_{l
  =1,\ldots,N}$. This means that the mode powers
$\{|\hat{a}_j^\eps(\omega,z) |^2 \}_{j=1,\ldots,N}$ converge in
distribution as $\eps \rightarrow 0$ to the diffusion Markov process
$\{|\hat{a}_j(\omega,z) |^2=P_j(\omega,z)\}_{j=1,\ldots,N}$, with
generator ${\cal L}_P$.

\item As we remarked before, the evanescent modes influence only the
  coefficient $\kappa_j(\om)$ which appears in ${\cal L}_\theta$ but
  not in ${\cal L}_P$. This means that the evanescent modes do not
  change the energy of the propagating modes in the limit $\eps \to
  0$. They also do not affect the coupling of the modes of the limit
  process, because $\kappa_j$ is in the diagonal part of
  \eqref{gendiff2T}. The only effect of the evanescent modes is a net
  dispersion (frequency dependent phase modulation) for each
  propagating mode.
\item  The generator ${\cal L}$ can also be written in the equivalent form 
\cite[Section 20.3]{book07}
\begin{eqnarray}
\nonumber
{\cal L} &=& 
\frac{1}{4} \sum_{{\scriptsize \begin{array}{c}j, l  = 1 \\
j \ne l \end{array}}} {\Gamma}_{jl}^{(c)}(\omega) \left(
A_{jl} \overline{A_{jl}} + 
\overline{A_{jl}} {A_{jl}} \right)  +
\frac{1}{2} \sum_{j,l=1}^N \Gamma^{(0)}_{jl}(\omega) A_{jj} \overline{A_{ll}} 
\\ 
&&
+ \frac{i}{4}
\sum_{{\scriptsize \begin{array}{c}j, l  = 1 \\
j \ne l \end{array}}} \Gamma^{(s)}_{jl}(\omega) (A_{jj} - A_{ll}) 
+  i 
\sum_{j=1 }^N \kappa_j(\om)  A_{jj} \, , 
\label{gendiffabis} 
\end{eqnarray}
in terms of the differential operators 
\begin{eqnarray}
A_{jl}&=& \hat{a}_j \frac{\partial}{\partial \hat{a}_l}
-\overline{\hat{a}_l} \frac{\partial}{\partial 
\overline{\hat{a}_j}} = - \overline{A_{lj}} \, .
\label{gendiffbbis}
\end{eqnarray}
Here the complex derivatives are defined in the standard way: if
$z=x+iy$, then $\partial_z=(1/2)(\partial_x-i\partial_y)$ and
$\partial_{\overline{z}}=(1/2)(\partial_x+i\partial_y)$.
\item 
 The coefficients of the second derivatives in (\ref{gendiffabis}) are
 homogeneous of degree two, while the coefficients of the first
 derivatives are homogeneous of degree one.  This implies that we can
 write closed ordinary differential equations in the limit $\eps \to 0$ for the
 moments of any order of $\{\hat a_j^\eps\}_{j = 1, \ldots, N}$.
\item 
Because 
\begin{equation} 
\label{eq:CONS}{\cal L} \left( \sum_{l=1}^N | \hat{a}_l |^2 \right)=0, 
\end{equation}
we have conservation of energy of the limit diffusion process. More
explicitly, the process is supported on the sphere in $\CC^N$ with
center at zero and radius $R_o$ determined by the initial
condition \[R_o^2 = \sum_{l=1}^N |{\hat{a}_{l,o}}(\omega)|^2.\] Since
${\cal L}$ is not self-adjoint on the sphere, the process is not
reversible.  But the uniform measure on the sphere is invariant, and
the generator is strongly elliptic.  From the theory of irreducible
Markov processes with compact state space, we know that the process is
ergodic and thus $\hat{\itbf a}(z)$ converges for large $z$ to the
uniform distribution over the sphere of radius $R_o$.  This can be
used to compute the limit distribution of the mode powers
$(|\hat{a}_j|^2)_{j=1,\ldots,N}$ for large $z$, which is the uniform
distribution over the set
\begin{equation} {\cal H}_N = \Big\{
\{P_j\}_{j=1,\ldots,N} , \, P_j \geq0, \, \sum_{j=1}^N P_j =R_o^2
\Big\} \, .  \label{eq:HN}
\end{equation} 
We carry out a more detailed analysis that is valid for any $z$ in the
next section.
\end{enumerate}

\subsubsection{Independence of the change of coordinates that flatten 
the boundaries} 
\label{sect:indc}
The coefficients \eqref{defgamma1}, \eqref{defgamma10} and
\eqref{defgamma1s} of the generator ${\cal L}$ have simple expressions
and are determined only by the covariance functions of the fluctuations
$\nu(z)$ and $\mu(z)$ and the boundary values of the derivatives of
the eigenfunctions $\phi_j(\om,\xi)$ in the unperturbed waveguide.
The dispersion coefficient $\kappa_j$ has a more complicated
expression \eqref{eq:kappa1}-\eqref{eq:kappa_e}, which involves
integrals of products of the eigenfunctions and their derivatives with
powers of $\xi$ or $X-\xi$. These factors in $\xi$ are present in our
change of coordinates
\begin{equation}
\ell^\eps(z,\xi) = B(z) + [T(z)-B(z)]\frac{\xi}{X} = \xi + \eps
\left[(X-\xi) \mu(z) +  \xi \nu(z)\right],
\label{eq:ell}
\end{equation}
so it is natural to ask if the generator ${\cal L}$ depends on the
change of coordinates. We show here that this is not the case.

Let $F^\eps(z,\xi) \in C^1\left([0,\infty) \times [0,X]\right)$ be a
  general change of coordinates satisfying 
\begin{equation}\label{as1}
F^{\eps}(z,\xi)=\left\{
\begin{array}{cll}
X(1+\eps\nu(z)) &\text{for} & \xi=X\\ \eps X \mu(z)
&\text{for} & \xi=0\;
\end{array}
\right.
\end{equation}
for each $\eps > 0$, and converging uniformly to the identity 
mapping as $\eps \to 0$, 
\begin{align}\label{as2}
\sup_{z\geq0}\sup_{\xi\in[0,X]}|F^{\eps}(z,\xi)-\xi| = O(\eps), \qquad 
\sup_{z\geq0}\sup_{\xi\in[0,X]}|\partial_{z}F^{\eps}(z,\xi)| = O(\eps).
\end{align} 
Note that \eqref{as2} is not restrictive in our context since
$(\mu(z),\nu(z))$ and their derivatives are uniformly bounded.  Define
the wavefield
\begin{equation}
\hat w(\om,\xi,z) = \hat p\left(\om,F^\eps(z,\xi),z\right),
\label{eq:w}
\end{equation}
and decompose it into the waveguide modes, as we did for $ \hat
u(\om,\xi,z) = \hat p\left(\om,\ell^\eps(z,\xi),z\right).  $ We have
the following result proved in appendix \ref{ap:coordc}.
\begin{theorem}
\label{thm.2}
The amplitudes of the propagating modes of the wave field \eqref{eq:w}
converge in distribution as $\eps \to 0$ to the same limit diffusion
as in Theorem \ref{propdiff}.
\end{theorem}
\subsubsection{The loss of coherence of the wave field}
\label{subseccoupledpower2}
From Theorem \ref{propdiff} and the expression \eqref{gendiffabis} of
the generator we get by direct calculation the following result for
the mean mode amplitudes.
\begin{proposition}
\label{prop.mean}
As $\eps \to 0$, $\EE[ \hat a_j^\eps(\om,z) ]$ converges to the expectation 
of the limit diffusion $\hat a_j(\om,z)$, given by 
\begin{equation}
\EE[\hat a_j(\om,z)] = \hat a_{j,o}(\om) \, \mbox{\em exp}\left \{
\Big[\frac{ \Gamma_{jj}^{(c)}(\om) - \Gamma_{jj}^{(0)}(\om)
  }{2}\Big] z + i \Big[ \frac{\Gamma^{(s)}_{jj}(\om)}{2} + \kappa_j(\om)
  \Big] z\right\}\, .
\label{eq:mean}
\end{equation}
\end{proposition}

\vspace{0.1in}
\noindent As we remarked before, $\Gamma_{jj}^{(c)} -
\Gamma_{jj}^{(0)}$ is negative, so the mean mode amplitudes
decay exponentially with the range $z$. Furthermore, we see from
\eqref{defgamma1} and \eqref{defgamma10} that $\Gamma_{jj}^{(c)} -
\Gamma_{jj}^{(0)}$ is the sum of terms proportional to $\left(
\partial_\xi \phi_j(X) \right)^2/\beta_j$ and $\left( \partial_\xi
\phi_j(0) \right)^2/\beta_j$. These terms increase with $j$, and they
can be very large when $j \sim N$. Thus, the mean amplitudes of the
high order modes decay faster in $z$ than the ones of the low order modes. 
We return to this point in
section \ref{sect:comparisson}, where we estimate the net attenuation
of the wave field in the high frequency regime $N \gg 1$.

That the mean field decays exponentially with range implies that the
wave field loses its coherence, and energy is transferred to its
incoherent part, the fluctuations. The incoherent part of the
amplitude of the $j-$th mode is $\hat a_j^\eps - \EE[\hat a_j^\eps ]$,
and its intensity is given by the variance $\EE[ |\hat a_j^\eps|^2] -
\left|\EE[\hat a_j^\eps]\right|^2$. The mode is incoherent if its mean
amplitude is dominated by the fluctuations, that is if
\[
\left[\EE[ |\hat a_j^\eps|^2] - \left|\EE[\hat
    a_j^\eps]\right|^2\right]^{1/2} \gg \left|\EE[\hat
  a_j^\eps]\right|.
\]
We know that the right hand side converges to \eqref{eq:mean} as $\eps
\to 0$. We calculate next the limit of the mean powers $\EE[ |\hat
  a_j^\eps|^2]$.

\subsubsection{Coupled power equations and equipartition of energy}

As we remarked in section \ref{sect:discuss}, the mode powers $|
\hat{a}^\eps_j(\omega,z)|^2$, for $j = 1, \ldots, N$, converge in
distribution as $\eps \rightarrow 0$ to the diffusion Markov process
$(P_j(\omega,z))_{j=1,\ldots,N}$ supported in the set \eqref{eq:HN},
and with infinitesimal generator ${\cal L}_P $.  We use this result to
calculate the limit of the mean mode powers
$$ {P}^{(1)}_j (\omega,z) = \EE [ P_j(\omega,z)]= \lim_{\eps
  \rightarrow 0} \EE [ | \hat{a}_j^\eps(\omega,z)|^2] \, .
$$ 
\begin{proposition}
\label{propmom20}
As $\eps \to 0$, $\EE[ | \hat{a}_j^\eps(\omega,z)|^2 ]$ converge to
${P}^{(1)}_j (\omega,z)$, the solution of the coupled linear system
\begin{equation}
\label{eqP1} 
\frac{d {P}^{(1)}_j}{dz } = \sum_{j=1}^N \Gamma_{jn}^{(c)}(\omega) \left(
     {P}^{(1)}_n- {P}^{(1)}_j \right) , \quad z >0\, ,
\end{equation}
with initial condition ${P}^{(1)}_j(\omega,z=0)= |
{\hat{a}_{j,o}}(\omega)|^2$, for $j=1,\ldots,N$.
\end{proposition}

\vspace{0.1in}
\noindent Matrix $\bGamma^{(c)}(\omega)$ is symmetric, with rows
summing to zero, by definition. Thus, we can can rewrite \eqref{eqP1}
in vector-matrix form
\begin{equation}
\frac{d {\itbf P}^{(1)}(z)}{dz} = \bGamma^{(c)}(\om) {\itbf P}^{(1)}(z),
\quad z > 0, ~ ~\mbox{and} ~ ~ {\itbf P}^{(1)}(0) = {\itbf P}^{(1)}_o,
\end{equation}
with ${\itbf P}^{(1)}(z) = \left({P}^{(1)}_1, \ldots,
{P}^{(1)}_n\right)^T$ and ${\itbf P}^{(1)}_o$ the vector with components
$| {\hat{a}_{j,o}}(\omega)|^2$, for $j = 1, \ldots, N$.  The solution
is given by the matrix exponential
\begin{equation}
{\itbf P}^{(1)}(z) = \mbox{exp} \left[ \bGamma^{(c)}(\om) z\right] 
{\itbf P}^{(1)}_o.
\label{eq:P1EXP}
\end{equation}
We know from \eqref{defgamma1} that the off-diagonal entries in
$\bGamma^{(c)}$ are not negative. If we assume that they are strictly
positive, which is equivalent to asking that the power spectral
densities of $\nu$ and $\mu$ do not vanish at the arguments
$\beta_j-\beta_l$, for all $j,l = 1, \ldots, N$, we can apply the
Perron-Frobenius theorem to conclude that zero is a simple eigenvalue
of $\boldsymbol{\Gamma}^{(c)}(\omega)$, and that all the other
eigenvalues are negative,
\[\Lambda_{N(\omega)}(\omega) \leq \cdots \leq
\Lambda_2(\omega)<0.\] This shows that as the range $z$ grows, the
vector ${\itbf P}^{(1)}(z)$ tends to the null space of
$\bGamma^{(c)}$, the span of the vector $(1,\ldots,1)^T$. That is to
say, the mode powers converge to the uniform distribution in the set
\eqref{eq:HN} at exponential rate
\begin{equation}
\label{eq:equip}
 \sup_{j=1,\ldots,N(\omega)} \Big| {P}^{(1)}_j (\omega,z) -
\frac{R_o^2(\omega)}{N(\omega)} \Big| \leq C e^{-|\Lambda_2(\omega)| z}
\, .
\end{equation}
 As $z \to \infty$, we have equipartition of energy among the
propagating modes.
\subsubsection{Fluctuations of the mode powers}
\label{secfluc}%
To estimate the fluctuations of the mode powers, we use again Theorem
\ref{propdiff} to compute the fourth order moments of the mode
amplitudes:
$$ {P}^{(2)}_{jl}(\omega,z) = \lim_{\eps \rightarrow 0} \EE \left[
|\hat{a}_j^\eps(\omega,z) |^2 |\hat{a}_l^\eps(\omega,z) |^2 \right]
=\EE [ P_j(\omega,z) P_l(\omega,z) ] \, .
$$ Using the generator ${\cal L}_P$, we get the following coupled
system of ordinary differential equations for limit moments 
\begin{eqnarray}
\frac{d {P}^{(2)}_{jj}}{dz} &=& \sum_{{\scriptsize \begin{array}{c} n
      = 1 \\ n \neq j \end{array} } }^N \Gamma_{jn}^{(c)} \left( 4
     {P}^{(2)}_{jn}-2 {P}^{(2)}_{jj} \right) \, , \nonumber \\ \frac{d
       {P}^{(2)}_{jl}}{dz} &=& - 2 \Gamma_{jl}^{(c)} {P}^{(2)}_{jl} +
     \sum_{n=1}^N \Gamma_{ln}^{(c)} \left( {P}^{(2)}_{jn} -
         {P}^{(2)}_{jl}\right) + \sum_{n= 1}^N \Gamma_{jn}^{(c)}
         \left( {P}^{(2)}_{ln} - {P}^{(2)}_{jl}\right),
         \ \ \ \ \ j\neq l \, , \quad z > 0, \label{eq:4thmom}
\end{eqnarray}
with initial conditions
\begin{equation}
{P}^{(2)}_{jl}(0)=|\hat{a}_{j,o}|^2|\hat{a}_{l,o}|^2.
\end{equation}
The solution of this system can be written again in terms of the
exponential of the evolution matrix.

It is straightforward to check that the function ${P}_{jl}^{(2)}
\equiv 1 +\delta_{jl}$ is a stationary solution of \eqref{eq:4thmom}.
Using the positivity of $\Gamma_{jl}^{(c)}$ for $j\neq l$, we conclude
that this stationary solution is asymptotically stable, meaning that
the solution ${P}_{jl}^{(2)}(z)$ converges as $z \rightarrow \infty$
to
$$
 {P}^{(2)}_{jl}(z)
\stackrel{z \rightarrow \infty}{\longrightarrow}
\left\{ \begin{array}{ll}
\displaystyle \frac{1}{N(N+1)} R_o^4 & \mbox{ if } j \neq l \, , \\
\displaystyle \frac{2}{N(N+1)} R_o^4 & \mbox{ if } j = l \, , 
\end{array}
\right.
$$ 
where $R_o^2 = \sum_{j=1}^N |\hat{a}_{j,o}|^2$.  This implies that
the correlation of $P_j(z)$ and $P_l(z)$ converges to $-1/(N-1)$ if $j
\neq l$ and to $(N-1)/(N+1)$ if $j=l$ as $z \to \infty$.  We see from
the $j \neq l$ result that if, in addition, the number of modes $N$
becomes large, then the mode powers become uncorrelated.  The $j=l$
result shows that, whatever the number of modes $N$, the mode powers
$P_j$ are not statistically stable quantities in the limit $z \to
\infty$, since
$$
\frac{{\rm Var} ( P_j(\omega,z) )}{\EE [P_j(\omega,z)]^2}
\stackrel{z \rightarrow \infty}{\longrightarrow}
\frac{N-1}{N+1}  \, .
$$

\section{Estimation of net diffusion}  
\label{sect:comparisson}
To illustrate the random boundary cumulative scattering effect  over long ranges, we 
quantify in this section  the diffusion coefficients
$\Gamma_{jl}^{(c)}$ and $\Gamma_{jl}^{(0)}$
in the generator ${\mathcal L}$ of the limit process. 
In particular, we calculate the mode-dependent net attenuation rate
\begin{equation}
{\mathcal K}_j(\om) = \frac{\Gamma_{jj}^{(0)}(\om) -
\Gamma_{jj}^{(c)}(\om)}{2}\, ,
\label{eq:EE2}
\end{equation}
that determines the coherent (mean) amplitudes as shown in 
\eqref{eq:mean}.  The attenuation rate gives the range scale over
which the $j-$th mode becomes essentially incoherent, because
equations \eqref{eq:mean} and \eqref{eq:P1EXP} give
\[
\frac{\left|\EE\left[ \hat a_j(\om,z) \right]\right|}{\sqrt{\EE\left[
      \left|\hat a_j(\om,z)\right|^2 \right] -\left|\EE\left[ \hat
      a_j(\om,z) \right]\right|^2}} \ll 1 \qquad \mbox{if} ~ 
z \gg  {\mathcal K}_j^{-1}.
\]
The reciprocal of the attenuation rate can therefore be interpreted as a scattering mean free path.
 The scattering mean free path is classically defined as the propagation distance 
 beyond which the wave loses its coherence \cite{rossum}. Here it is mode-dependent.

Note that the attenuation rate ${\mathcal K}_j(\om)$ is the sum of two terms. The first one involves 
the phase diffusion coefficient $\Gamma_{jj}^{(0)}$ in the generator ${\mathcal L}_\theta$, and determines
the range scale over which the cumulative random phase of the amplitude $\hat a_j$ becomes significant, 
thus  giving exponential damping of the expected field  $\EE[\hat a_j]$.  
The second term is the mode-dependent  energy exchange rate
\begin{equation}
{\mathcal J}_j(\om) = - \frac{
\Gamma_{jj}^{(c)}(\om)}{2}\, , 
\label{eq:EE2J}
\end{equation}
given  by the power diffusion coefficients in the generator ${\mathcal L}_P$.
Each waveguide mode can be associated with a direction of incidence at the unperturbed  boundary,
and energy is exchanged between modes when they scatter, because 
of the   fluctuation of the angles of incidence at the random boundaries. 
We can interpret the reciprocal of the energy exchange rate
as a transport mean free path, which  is classically defined as the
distance beyond which the wave forgets its initial direction \cite{rossum}.

The third important length scale is the equipartition distance $1/|\Lambda_2(\om)|$, defined in terms 
of the second largest eigenvalue of the matrix $\bGamma^{(c)}(\om)$. It is the distance over 
which the energy becomes uniformly distributed over the modes, independently  of the 
initial excitation at the source, as shown in equation \eqref{eq:equip}. 

\subsection{Estimates for a waveguide with constant wave speed}
To give sharp estimates of ${\mathcal K}_j$ and ${\mathcal J}_j$ for
$j = 1, \ldots, N$, we assume in this section a waveguide with
constant wave speed $c(\xi) = c_o$ and a high frequency regime $N \gg
1$. Note from \eqref{defgamma1b} that the magnitude of
$\Gamma_{jj}^{(c)}$ depends on the rate of decay of the power spectral
densities $\hat \cR_\nu(\beta)$ and $\hat \cR_\mu(\beta)$ with respect
to the argument $\beta$.  We already made the assumption
\eqref{validforward} on the decay of the power spectral densities, in
order to justify the forward scattering approximation.  In particular, we assumed that 
$\hat \cR_\nu(\beta)\simeq \hat \cR_\mu(\beta) \simeq 0$ for all $\beta \geq 2 \beta_N$.
Thus, for a given mode index $j$, we expect large terms in the sum in 
\eqref{defgamma1b} for indices $l$ satisfying
\begin{equation}
|\beta_j - \beta_l| \lesssim 2 \beta_N = \frac{2 \pi}{X} \sqrt{2 \alpha N}  ,
\label{eq:EE9}
\end{equation}
where we used the definition 
\begin{equation}
\beta_j(\om) = \frac{\pi}{X} \sqrt{ (N + \alpha)^2 - j^2}, \quad j =1,
\ldots, N, \quad \mbox{and} \quad \frac{kX}{\pi} = N + \alpha, \quad
\mbox{for} ~~ \alpha \in (0,1)\, .
\label{eq:EE5}
\end{equation}
Still, it is difficult to get a precise estimate of
$\Gamma_{jj}^{(c)}$ given by \eqref{defgamma1b}, 
unless we make further assumptions on $\cR_\nu$ and
$\cR_\mu$.
For the calculations in this section we take the Gaussian
covariance functions
\begin{equation}
\cR_\nu(z) = 
\mbox{exp}\left(-\frac{z^2}{2 \ell^2_\nu} \right) \quad \mbox{and} \quad
\cR_\mu(z) = 
\mbox{exp}\left(-\frac{z^2}{2 \ell^2_\mu} \right) \, ,
\label{eq:Gaussian}
\end{equation}
and we take for convenience equal correlation lengths $\ell_\nu = \ell_\mu = \ell\, .$  The power
spectral densities are
\begin{equation}
\hat \cR_\nu(\beta) = \hat \cR_\mu(\beta) =  \sqrt{2 \pi} \, \ell\,  \mbox{exp} \left(
-\frac{\beta^2 \ell^2}{2} \right) \, ,
\label{eq:PSGaussian}
\end{equation}
and they are negligible for $ \beta \geq {3}/{\ell}$.   Since $N = \left \lfloor {k
  X}/{\pi} \right \rfloor$, we see that \eqref{eq:EE9} becomes
\begin{equation}
|\beta_j-\beta_l| \leq \frac{3}{\ell} \lesssim \frac{2 \pi}{X} \sqrt{2
  \alpha N} \quad \mbox{or equivalently, } \quad k \ell \gtrsim \frac{3}{2
  \sqrt{2 \alpha}} \sqrt{N} \gg 1\, .
\label{eq:kllarge}
\end{equation}
Thus, assumption \eqref{validforward} amounts to having correlation
lengths that are larger than the wavelength. The attenuation and exchange energy rates \eqref{eq:EE2} and \eqref{eq:EE2J}
are estimated in detailed in Appendix \ref{ap:estim}. We summarize the results in 
the following proposition, in the case\footnote{The case $k \ell \gtrsim N$ is also discussed in Appendix \ref{ap:estim}.} 
\begin{equation}
\sqrt{N} \lesssim k \ell 
\ll N.
\label{eq:assumekell}
\end{equation}

\begin{proposition}
\label{prop.estim}
The attenuation  rate ${\mathcal K}_j(\om)$ increases monotonically
with the mode index $j$. 
The energy exchange rate ${\mathcal J}_j(\om)$ increases monotonically
with the mode index $j$ up to the high modes of order $N$ where it can decay if $k\ell \gg \sqrt{N}$. 
For the low order modes we have
\begin{equation}
{\mathcal J}_j(\om) X 
\approx 
{\mathcal K}_j(\om) X \sim (k\ell)^{-1/2} 
 , \quad j \sim 1\, .
\label{eq:estAtt1}
\end{equation}
For the intermediate modes we have
\begin{equation}
{\mathcal J}_j(\om) X  
\approx
{\mathcal K}_j(\om) X \sim N^2 \frac{(j/N)^3}{\sqrt{1-(j/N)^2}}  
 , \quad 1 \ll j \ll N \, .
\label{eq:estAtt1med}
\end{equation}
For the high order modes we have 
\begin{equation}
{\mathcal J}_j(\om) X \sim \frac{N^3}{k \ell}  , \quad \quad  {\mathcal K}_j(\om) X \sim  k \ell N^2\, ,  \quad j \sim N\, ,
\label{eq:estAttN}
\end{equation}
for $k\ell \sim \sqrt{N}$, but when $k \ell \gg \sqrt{N}$, 
\begin{equation}
{\mathcal J}_j(\om) X \ll {\mathcal K}_j(\om) X \sim  k \ell N^2\, , \quad j \sim N\, .
\end{equation}
\end{proposition}
  
The results summarized in Proposition \ref{prop.estim} show that
scattering from the random boundaries has a much stronger effect on
the high order modes than the low order ones. This is intuitive,
because the modes with large index bounce more often from the
boundaries. The damping rate ${\mathcal K}_j$ is very large, of order
$N^2 k\ell$ for $j \sim N$, which means that the amplitudes of these
modes become incoherent quickly, over scaled\footnote{Recall from
section \ref{sec:chap21_propmatr} that the range is actually $z/
\eps^2$.}  ranges $z \sim X N^{-2} (k\ell)^{-1} \ll X$.  The modes with index $j
\sim 1$ keep their coherence over ranges $z = O(X)$, because their
mean amplitudes are essentially undamped ${\mathcal K}_j X \ll 1$ for
$j \sim 1$. 
However, the modes lose their coherence eventually,
because the damping becomes visible at longer ranges $z > X (k\ell)^{1/2}$.

Note  that the scattering mean free paths and the transport mean free paths are approximately the same for the low 
and intermediate index modes, but not for the high ones. The energy exchange rate for the high order modes may be much smaller than the 
attenuation rate in high frequency regimes with $k \ell \gg \sqrt{N}$.  These modes  
reach the boundary many times over a correlation length, at almost the same angle of incidence, so the exchange of energy is 
not efficient and it occurs only between neighboring modes.  There is however a significant cumulative random phase 
in $\hat a_j$ for $j \sim N$, given by the addition of the correlated phases gathered over the multiple scattering events.
This significant phase causes the loss of coherence of the amplitudes of the high order modes, the strong damping of   $\EE[\hat a_j]$.

Note also  that a direct calculation\footnote{By direct calculation we mean numerical calculation of the 
eigenvalue. We find that for $N \geq 20$ and for $k \ell \gtrsim \sqrt{N}$,  $|\Lambda_2(\om)| \approx |\Gamma_{11}^{(c)}(\om)|$
with a relative error that is less than $1\%$.} of the second largest eigenvalue of $\Gamma^{(c)}(\om)$ gives that 
\[
|\Lambda_2(\om)| \approx |\Gamma_{11}^{(c)}(\om)| \sim (k \ell)^{-1/2}.
\]
Thus, the equipartition distance is similar to  the  scattering mean free path of the first mode. 
This mode can travel longer distances than the others before it loses its coherence, but once that happens,  the waves have entered the 
 equipartition regime, where the energy is uniformly distributed among all the modes. The waves forget the initial condition at the source. 

\subsection{Comparison with waveguides with internal random inhomogeneities}
When we compare the results in Proposition \ref{prop.estim} with those
in \cite[Chapter 20]{book07} for random waveguides with interior
inhomogeneities but straight boundaries, we see that even though the
random amplitudes of the propagating modes converge to a Markov
diffusion process with the same form of the generator as
\eqref{gendiffabis}, the net effects on coherence and energy exchange are different in terms of their dependence
with respect to the modes.

Let us look in detail at the attenuation rate that determines the
range scale over which the amplitudes of the propagating modes lose
coherence. To distinguish it from \eqref{eq:EE2}, we denote the
attenuation rate by $\widetilde {\mathcal K}_j$ and the energy exchange rate by 
$\widetilde {\mathcal J}_j$, and recall from 
\cite[Section 20.3.1]{book07} that they are given by 
\begin{equation}
\label{Ktilde}
\widetilde {\mathcal K}_j = \frac{k^4 \hat \cR_{jj}(0)}{8 \beta_j^2} +\widetilde {\mathcal J}_j 
\, , \quad \quad  \quad 
\widetilde {\mathcal J}_j = 
\sum_{{\scriptsize \begin{array}{c}l = 1 \\ l \ne
      j \end{array}} }^N
\hspace{-0.05in}\frac{k^4}{8 \beta_j \beta_l } \hat
\cR_{jl}\left(\beta_j-\beta_l\right)\, .
\end{equation}
Here $\hat \cR_{jl}(z)$ is the Fourier transform (power spectral
density) of the covariance function $\cR_{jl}(z)$ of the stationary
random processes
\[
C_{jl}(z) = \int_0^X d x \, \phi_j(x)\phi_l(x) \nu(x,z)\,   , 
\]
the projection on the eigenfunctions of the random fluctuations
$\nu(x,z)$ of the wave speed. 

For our comparison we assume isotropic, stationary fluctuations with
mean zero and Gaussian covariance function
\[
\cR(x,z) = \EE\left[ \nu(x,z) \nu (0,0) \right] = e^{-\frac{x^2+z^2}{2 \ell^2}}\,,
\] 
so the power spectral densities are 
\begin{equation}
\hat \cR_{jl}(\beta) \approx \frac{\pi \ell^2}{X} e^{- \frac{(k
    \ell)^2}{2}\left(\frac{X \beta}{\pi N}\right)^2} \left[e^{ -
    \frac{(k \ell)^2}{2}\left(\frac{j}{N}-\frac{l}{N}\right)^2} + e^{
    - \frac{(k \ell)^2}{2}\left(\frac{j}{N}+\frac{l}{N}\right)^2} +
 \delta_{jl} \right] \, .
\end{equation}
Thus, \eqref{Ktilde} becomes
\begin{eqnarray*}
\widetilde {\mathcal K}_j &=& \frac{\pi (k \ell)^2}{8 X} \frac{2
  + e^{-2 (k \ell)^2 (j/N)^2}}{ \left(1+\alpha/N\right)^2 - (j/N)^2} +
\widetilde {\mathcal J}_j   ,  \\
\widetilde {\mathcal J}_j &=& \frac{\pi (k \ell)^2}{8 X}  
\sum_{{\scriptsize \begin{array}{c}l = 1 \\ l \ne
      j \end{array}} }^N
\hspace{-0.05in} \frac{e^{-\frac{(k \ell)^2}{2} \left[
      \sqrt{\left(1+\alpha/N\right)^2 - (j/N)^2} -
      \sqrt{\left(1+\alpha/N\right)^2 - (l/N)^2}\right]^2}}{
  \sqrt{\left[\left(1+\alpha/N\right)^2 -
      (j/N)^2\right]\left[\left(1+\alpha/N\right)^2 -
      (l/N)^2\right]}} \left[e^{ -
      \frac{(k \ell)^2}{2}\left(\frac{j}{N}-\frac{l}{N}\right)^2} +
    e^{ - \frac{(k
        \ell)^2}{2}\left(\frac{j}{N}+\frac{l}{N}\right)^2}\right] 
 \,  ,
\end{eqnarray*}
and their estimates can be obtained using the same techniques as in Appendix 
\ref{ap:estim}. We give here the results when $k\ell$ satisfies (\ref{eq:assumekell}).
For the low order modes  we have
\begin{eqnarray*}
\widetilde {\mathcal K}_j  X & \approx & \frac{\pi (k \ell)^2}{8} \left[ 2
  + e^{-2(k\ell)^2/N^2} + \frac{N \sqrt{\pi/2}}{k \ell} \right] \sim 
 \left[  (k \ell)^2 + N \, k \ell 
  \right]  \sim 
 N \, k \ell
\gtrsim N^{3/2}   , \quad j \sim 1, \\
\widetilde {\mathcal J}_j  X &\approx & \frac{\pi (k \ell)^2}{8}  \frac{N \sqrt{\pi/2}}{k \ell}   \sim 
 N \, k \ell
  \gtrsim   N^{3/2} , \quad j \sim 1,   
\end{eqnarray*}
and for the high order modes we have
\begin{eqnarray*}
\widetilde {\mathcal K}_j X &\approx& \frac{\pi N (k \ell)^2}{8  \alpha}
\left[1 + \frac{\sqrt{\pi} N}{2 \sqrt{2} k \ell} \right] =
 \left[  N (k \ell)^2  +   N^2 k
  \ell \right] \sim N^2  k\ell \gtrsim N^{5/2} , \quad j
\sim N, \\
\widetilde {\mathcal J}_j X &\approx& \frac{\pi N (k \ell)^2}{8  \alpha}
  \frac{\sqrt{\pi} N}{2 \sqrt{2} k \ell} =
 N^2 k
  \ell   \gtrsim  N^{5/2},  \quad j
\sim N.
\end{eqnarray*}
Thus, we see that in waveguides with internal random inhomogeneities
the low order modes lose coherence much faster than in waveguides
with random boundaries. Explicitly, coherence is lost over scaled
ranges
\[
z \lesssim X\,   N^{-3/2}  \ll X.
\]
The high order modes, with index $j \sim N$, lose coherence over
the range scale
\[
z \lesssim X \,  N^{-5/2}  \ll X.
\]
Moreover, the main mechanism for the loss of coherence is the exchange of energy 
between neighboring modes. That is to say, the transport mean free path is 
equivalent to the scattering mean free path  for all the modes in  random waveguides with interior
inhomogeneities.
Finally, direct (numerical) calculation shows that 
\[
O\left((k \ell)^{-2}\right) \leq \frac{|\Lambda_2|}{|\widetilde {\mathcal J}_1|} \leq O \left( (k \ell)^{-3/2}\right)\, ,
\]
so the equipartition distance is 
larger by a factor of at least $O\left(N^{3/4}\right)$  than the scattering or transport mean free path.

\section{Mixed boundary conditions}
\label{sect:mixed}
Up to now we have described in detail the wave field in waveguides
with random boundaries and Dirichlet boundary conditions
\eqref{eq:Dirichlet}. In this section we extend the results to the
case of mixed boundary conditions \eqref{eq:mixed}, with Dirichlet
condition at $x = B(z)$ and Neumann condition at $x = T(z)$.  All
permutations of Dirichlet/Neumann conditions are of course possible,
and the results can be readily extended.

Similar to what we stated in section \ref{sect:homog}, the operator
$\partial_x^2 + \omega^2 c^{-2}(x)$ acting on functions in $(0,X)$,
with Dirichlet boundary condition at $x=0$ and Neumann boundary
condition $x=X$, is self-adjoint in $L^2(0,X)$.  Its spectrum is an
infinite number of discrete eigenvalues $\lambda_j(\omega)$, for
$j=1,2,\dots$, and we sort them in decreasing order. There is a finite
number $N(\om)$ of positive eigenvalues and an infinite number of
negative eigenvalues.  We assume as in section \ref{sect:homog} that
$N(\om) = N$ is constant over the frequency band, and that the
eigenvalues are simple.  The modal wavenumbers are as before, $
\beta_j(\om) = \sqrt{|\lambda_j(\om)|}\, .  $ The eigenfunctions
$\phi_j(\omega,x)$ are real and form an orthonormal set.

For example, in the case of a constant wave speed $c(x) = c_o$, we have 
\begin{equation}
\lambda_j = k^2 - \left[ \frac{(j-1/2) \pi}{X}\right]^2, \qquad \phi_j(x) =
\sqrt{\frac{2}{X}} \sin \left( \frac{(j-1/2) \pi x}{X} \right), \qquad
j = 1, 2, \ldots\, , 
\end{equation}
and the number of propagating modes is given by $N = \left \lfloor
\frac{k X}{\pi} + \frac{1}{2} \right \rfloor.$

\subsection{Change of Coordinates}
We proceed as before and straighten the boundaries using a change of
coordinates that is slightly more complicated than before, due to the
Neumann condition at $x = T(z)$, where the normal is along the vector
$(1,-T'(z))$. We let 
\begin{equation}
p(t,x,z) = u\big( t , \cX(x,z), \cZ(x,z) \big)\, ,
\end{equation}
where 
\begin{eqnarray}
\cX(x,z) &=& X\frac{ x-B(z)}{T(z)-B(z)} \, ,\label{eq:NX}\\ \cZ(x,z) &=& z + x
T'(z) + Q(z)\, , \quad \quad Q(z)= - \int_0^z ds \, T(s) T''(s) \, .
\label{eq:NZ}
\end{eqnarray}
In the new frame we get that $ \xi = \cX(x,z) \in [0,X]$, with 
Dirichlet condition at $\xi = 0$
\begin{equation}
 {u}(t,\xi=0,\zeta) =0 \, .
\end{equation}
For the Neumann condition at $\xi = X$ we use the chain rule, and rewrite 
\[\partial_\nu p(t, x=T(z),z ) = \big[\partial_x -T'(z) \partial_z 
  \big] p(t,x=T(z),z )=0\, ,
\]
as
\begin{eqnarray*}
 && \partial_\xi u (t,\xi=X,\zeta=\cZ(T(z),z) ) \big[ -\partial_x \cX +T'(z)
  \partial_z \cX \big](x=T(z),z) + \\ && \partial_\zeta u
  (t,\xi=X,\zeta=\cZ(T(z),z) ) \big[ -\partial_x \cZ +T'(z) \partial_z \cZ
  \big](x=T(z),z) =0\, .  
\end{eqnarray*}
This is the standard Neumann condition
\begin{equation}
 \partial_\xi {u}(t ,\xi=X,\zeta)=0,
\end{equation}
because
$$ \big[ -\partial_x \cZ +T'(z) \partial_z \cZ \big](x=T(z),z) =
-T'(z) +T'(z) \big[ 1+T(z)T''(z) +Q'(z) \big]=0 \, ,
$$
and 
\[
\big[ -\partial_x \cX + T'(z) \partial_z \cX\big](x=T(z),z) = -
\frac{X + \left[T'(z)\right]^2}{T(z)-B(z)} \ne 0\, .
\]

Now, the method of solution is as before. Using that $\eps$ is small,
we obtain a perturbed wave equation for $\hat{u}$, which we expand as
\begin{eqnarray}
{\mathcal L}_0  \hat{u} + \eps {\mathcal L}_1 \hat{u} 
 + \eps^2 {\mathcal L}_2 \hat{u} =O(\eps^3) ,
\label{eq:pertw2n}
\end{eqnarray}
with leading order operator 
\[
{\mathcal L}_0 = \partial_\zeta^2 + \partial_\xi^2 +\omega^2 /c^{2}
  (\xi)\, , 
\]
and perturbation 
\begin{eqnarray}
{\mathcal L}_1 = -2 (\nu-\mu) \partial_\xi^2 +2 (X- \xi) (\nu'-\mu')
  \partial_{\zeta\xi} -2 X(X-\xi ) \nu'' \partial_\zeta^2 -X(X-\xi)
  \nu''' \partial_\zeta - \\ \nonumber \big[X \mu'' + \xi
  (\nu''-\mu'')\big] \partial_\xi + \omega^2 (\partial_\xi c^{-2}(\xi)
  ) \big[X\mu +(\nu-\mu) \xi \big] \, .
\end{eqnarray}

\subsection{Coupled Amplitude Equations}
\label{sec:CAEn}%
We proceed as in section \ref{sect:wavedec}.  We find that the complex
 mode amplitudes satisfy \eqref{eq:WD3}-\eqref{eq:WD4} with $\zeta$
 instead of $z$, where the $\zeta$-dependent coupling coefficients are
\begin{eqnarray}
  \label{def:Cjln}
C_{jl}^\eps (\zeta) &=& 
\eps C_{jl}^{(1)} (\zeta)
+
\eps^2 C_{jl}^{(2)} (\zeta)
+ O(\eps^3) \, ,
\\
\nonumber
C_{jl}^{(1)} (\zeta) &=&
 c_{\nu,jl} \nu(\zeta) +  i \beta_l  d_{\nu,jl}  \nu'(\zeta)+ e_{\nu,jl} \nu''(\zeta)
 +  i \beta_l  f_{\nu,jl}  \nu'''(\zeta)   \\
&&+
 c_{\mu,jl} \mu(\zeta) +  d_{\mu,jl} \big(2i \beta_l \mu'(\zeta) +\mu''(\zeta)\big) 
   \, , 
\end{eqnarray}
with
\begin{eqnarray}
\label{def:cnun}
c_{\nu,jl} &=& \frac{1}{2 \sqrt{\beta_j \beta_l}} \Big[ \Big(
\frac{\omega^2}{c(X)^2}- \beta_l^2\Big) \phi_j(X) \phi_l(X)
+(\beta_j^2-\beta_j^2) \int_0^X d \xi \, \xi \phi_l \partial_\xi\phi_j
\Big] \, ,\\ d_{\nu,jl} &=& \frac{1}{2 \sqrt{\beta_j \beta_l}} \Big[ 2
\int_0^2 d \xi \, (X- \xi) \phi_j \partial_\xi\phi_l \Big]\, ,\\
e_{\nu,jl} &=& \frac{1}{2 \sqrt{\beta_j \beta_l}} \Big[ - \int_0^X d
\xi \, (X- \xi) \phi_j \xi \partial_\xi\phi_l +2 \beta_l^2 \int_0^X d
\xi  (X-\xi) \phi_j \phi_l \Big]\, ,\\ f_{\nu,jl} &=& \frac{1}{2
\sqrt{\beta_j \beta_l}} \Big[ - \int_0^X d \xi \, (X- \xi) \phi_j
\phi_l \Big]\, ,
 \end{eqnarray}
and coefficients $c_{\mu,jl} $ and $d_{\mu,jl}$ defined by
\eqref{eq:cmu} and \eqref{eq:dmu}. Similar formulas hold for
$C^{(2)}_{jl}(\zeta)$.

In the following we neglect for simplicity the evanescent modes, which
only add a dispersive (frequency dependent phase modulation) net effect in the problem. These modes can be
included in the analysis using a similar method to that in section
\ref{sect:elim_evanesc}.

\subsection{The Coupled Mode Diffusion Process}
\label{subseccoupledpowern}%

As we have done in section \ref{sect:diffusion}, we study under the
forward scattering approximation the long range limit of the forward
propagating mode amplitudes.

First, we give a lemma which shows that the description of the wave
field in the variables $(x,z)$ or $(\xi,\zeta)$ is asymptotically
equivalent.
\begin{lemma}
We have uniformly in $x$
$$ \cX \left( x,\frac{z}{\eps^2}\right) - x \stackrel{\eps \to
  0}{\longrightarrow} 0,\quad \quad \cZ\left( x,\frac{z}{\eps^2}\right)
-\frac{z}{\eps^2} - \EE[ \nu'(0)^2] z \stackrel{\eps \to
  0}{\longrightarrow} 0 \mbox{ in probability} \, .
$$
\end{lemma}
\begin{proof}
The convergence of $\cX$ to $x$ is evident from definitions
\eqref{eq:NX} and \eqref{eq:defBT}. Moreover, \eqref{eq:NZ} gives 
\begin{eqnarray*}
\cZ\left( x,\frac{z}{\eps^2}\right) -\frac{z}{\eps^2} &=& x \eps X \nu'
  \left(\frac{z}{\eps^2}\right) -\eps X^2\int_0^{\frac{z}{\eps^2}}
  (1+\eps \nu(s) ) \nu''(s) ds\, ,  
\end{eqnarray*}
and integrating by parts and using the assumption that the
fluctuations vanish at $z = 0$, we get 
\begin{eqnarray*}
\cZ\left( x,\frac{z}{\eps^2}\right) -\frac{z}{\eps^2} &=& \eps X
\left[ (x-X) \nu' \left(\frac{z}{\eps^2}\right)- \eps
\nu\left(\frac{z}{\eps^2}\right) \nu' \left(\frac{z}{\eps^2}\right)
\right] + \eps^2 \int_0^{\frac{z}{\eps^2}} \left[\nu'(s)\right]^2ds \, .
\end{eqnarray*}
The first term of the right-hand side is of order $\eps$ and the
second term converges almost surely to $\EE[ \nu'(0)^2] z$ which gives
the result.
\end{proof}

The diffusion limit is similar to that in section \ref{subseccoupledpower}, 
and the result is as follows. 
\begin{proposition}
The complex mode amplitudes $(\hat{a}_j^\eps(\omega,\zeta)
)_{j=1,\ldots,N}$ converge in distribution as $\eps \rightarrow 0$ to
a diffusion Markov process process $(\hat{a}_j(\omega,\zeta)
)_{j=1,\ldots,N}$.  Writing
$$ \hat{a}_j(\omega,\zeta) = P_j(\omega,\zeta)^{1/2} e^{i
\phi_j(\omega,\zeta)}, \quad j=1,\ldots,N,
$$ the infinitesimal generator of the limiting diffusion process 
\[
{\mathcal L} = {\mathcal L}_P + {\mathcal L}_\theta
\]
is of the form (\ref{gendiffa}), but with different expressions of the
coefficients given below.
\end{proposition}

\vspace{0.1in}
\noindent The coefficients $\Gamma^{(c)}_{jl}$ in ${\mathcal L}_P$ are
given by
\begin{equation}
\Gamma_{jl}^{(c)}(\omega) = \hat \cR_\mu\left(\beta_j-\beta_l\right)
Q_{\nu,jl}^2 + \hat \cR_\mu\left(\beta_j-\beta_l\right) Q_{\mu,jl}^2
\quad \mbox{ if } j \neq l \, ,
\end{equation}
where
\begin{eqnarray}
\nonumber Q_{\nu,jl} &=& c_{\nu,jl} + d_{\nu,jl}
\beta_l(\beta_l-\beta_j) -(\beta_l-\beta_j)^2 \big[ e_{\nu,jl} +
f_{\nu,jl} \beta_l(\beta_l-\beta_j) \big] \\ &=& \frac{X}{2
\sqrt{\beta_j \beta_l}}\left[ \frac{\omega^2}{c(X)^2}- \beta_l\beta_j
\right] \phi_j(X) \phi_l(X) \, , \\ \nonumber Q_{\mu,jl} &=&
c_{\mu,jl} +d_{\mu,jl} (\beta_l^2 -\beta_j^2) = \frac{X}{2
\sqrt{\beta_j \beta_l}} \partial_\xi \phi_j(0) \partial_\xi \phi_l(0)
\, .
\end{eqnarray}
The coefficients in ${\mathcal L}_\theta$ are similar, 
\begin{equation}
\Gamma_{jl}^{(0)}(\omega) = \hat \cR_\mu(0) Q_{\nu,jl}^2 + \hat
\cR_\mu(0) Q_{\mu,jl}^2 \quad \forall j, l \, ,
\end{equation}
and 
\begin{equation}
\Gamma_{jl}^{(s)}(\omega) = \gamma_{\nu,jl} Q_{\nu,jl}^2 +
\gamma_{\mu,jl} Q_{\mu,jl}^2 \quad \mbox{ if } j \neq l \, ,
\end{equation}
with $\gamma_{\nu,jl}$ and $\gamma_{\mu,jl}$ defined by
\eqref{eq:gammanumu}.

We find again that these effective coupling coefficients depend only
on the behaviors of the mode profiles close to the boundaries.  In the
case of Dirichlet boundary conditions, the mode coupling coefficient
$\Gamma_{jl}^{(c)}(\omega)$ depends on the value of $ \partial_\xi
\phi_j \partial_\xi \phi_l$ at the boundaries.  In the case of Neumann
boundary conditions, the mode coupling coefficient
$\Gamma_{jl}^{(c)}(\omega)$ depends on the value of $ \phi_j(X)
\phi_l(X)$.

Given the generator, the analysis of the loss of coherence, and of the
mode powers is the same as in sections
\ref{subseccoupledpower2}-\ref{secfluc}.

\section{Summary}
\label{sect:summary}
In this paper we obtain a rigorous quantitative analysis of wave propagation in
two dimensional waveguides with random and stationary fluctuations of
the boundaries, and either Dirichlet or Neumann boundary conditions.
The fluctuations are small, of order $\eps$, but their effect becomes
significant over long ranges $z/\eps^2$. We carry the analysis in
three main steps: First, we change coordinates to straighten the
boundaries and obtain a wave equation with random coefficients.
Second, we decompose the wave field in propagating and evanescent
modes, with random complex amplitudes satisfying a random system
of coupled differential equations. We analyze the evanescent modes
and show how to obtain a closed system of differential equations for
the amplitudes of the propagating modes. In the third step we analyze
the amplitudes of the propagating modes in the long range limit, and
show that the result is independent of the particular choice of the
change of the coordinates in the first step. The limit process is a
Markov diffusion with coefficients in the infinitesimal generator
given explicitly in terms of the covariance of the boundary
fluctuations. Using this limit process, we quantify mode by mode the
loss of coherence and the exchange (diffusion) of energy between modes
induced by scattering at the random boundaries.

The long range diffusion limit is similar to that in random waveguides
with interior inhomogeneities and straight boundaries, in the sense
that the infinitesimal generators have the same form. However, the net
scattering effects are very different. We quantify them explicitly in a 
high frequency regime, in the case of a constant wave speed, and compare
the results with those in waveguides with interior random inhomogeneities. 
In particular, we estimate three important length scales: the scattering mean free path, 
the transport mean free path and the equipartition distance. The first two give the distances 
over which the waves lose their coherence and  forget their direction, respectively.  
The last is the distance over which the cumulative scattering distributes the energy uniformly 
among the modes, independently of the initial conditions at the
source. 

We obtain that in waveguides with random boundaries the lower order modes
have a longer scattering mean free path, which is comparable to the transport mean free 
path and, remarkably to the equipartition distance. The high order modes 
lose coherence rapidly, they have a short scattering mean free path, and do not exchange energy efficiently with the other modes. 
They  also have a transport mean free path that exceeds the scattering mean free path.
In contrast, in waveguides with interior random inhomogeneities, all the modes lose 
their coherence over much shorter distances than in waveguides with random boundaries.   
Moreover, the main mechanism of loss of coherence is the exchange of energy with 
the nearby modes, so  the scattering mean free paths and the transport mean free paths are similar for all the modes. 
Finally, the equipartition distance is much longer than the distance over which all the modes 
lose their coherence.

These results are useful in applications such as imaging with remote
sensor arrays. Understanding how the waves lose coherence is essential
in imaging, because it allows the design of robust methodologies that
produce reliable, statistically stable images in noisy environments
that we model mathematically with random processes. An example of a
statistically stable imaging approach guided by the theory in random
waveguides with internal inhomogeneities is in \cite{borcea}.

\section*{Acknowledgments}  
The work of R. Alonso was partially supported by the Office of Naval
Research, grant N00014-09-1-0290 and by the National Science
Foundation Supplemental Funding DMS-0439872 to UCLA-IPAM.  The work of
L. Borcea was partially supported by the Office of Naval Research,
grant N00014-09-1-0290, and by the National Science Foundation, grants
DMS-0907746, DMS-0934594.

\appendix
\section{Proof of Lemma \ref{lem.1}}
\label{sect:Proof}
The proof given here relies on explicit estimates of the series in
(\ref{eq:E6}), obtained under the assumption that the background speed
is constant $c(\xi) = c_o$.
We rewrite (\ref{eq:E6}) as
\begin{equation}
\left[\Psi \hat{\itbf v}\right](\om,z) = \left[\Psi_1 \hat{\itbf
v}\right](\om,z) + \left[\Psi_2 \hat{\itbf
v}\right](\om,z)
\label{eq:PE1}
\end{equation}
with linear integral operators $\Psi_1$ and $\Psi_2$ defined
component wise by
\begin{eqnarray}
\big[ \Psi_1\hat{\itbf v} \big]_{j}(\om,z) &=& \sum_{l= N+1}^\infty
\frac{1}{2\beta_{j}}\int^{\infty}_{-\infty}
(M^{\eps}_{jl}-\partial_{z}Q^{\eps}_{jl})(z+s)\hat{v}_{l}(\om,z+s)
e^{-\beta_{j}|s|}ds  ,
\label{eq:Psi1}  \\
\big[\Psi_2 \hat{\itbf v} \big]_{j}(\om,z) 
&=& 
\sum_{l = N+1}^\infty
\frac{1}{2}\int^{\infty}_{-\infty} Q^{\eps}_{jl}(z+s)
\hat{v}_{l}(\om, z+s)e^{-\beta_{j}|s|}ds.
\label{eq:Psi2}
\end{eqnarray}
The coefficients have the explicit form
\begin{eqnarray}
M_{jl}^\eps(z) &=& \left\{ 2 \left[\nu(z)-\mu(z)\right]
\left(\frac{\pi j}{X}\right)^2 + \frac{\nu''(z)-\mu''(z)}{2} \right\}
\delta_{jl} + (1-\delta_{jl}) \left[ \nu''(z)-\mu''(z)\right]
\frac{2lj}{j^2-l^2} - \nonumber \\ && (1-\delta_{jl}) \nu''(z)
\frac{2lj}{j^2-l^2} \left[ 1 - (-1)^{l+j} \right] + O(\eps)  ,
\label{eq:PE2}\\
Q_{jl}^\eps(z) &=&  \left[\nu'(z) - \mu'(z) \right] \delta_{jl} +
(1-\delta_{jl}) \left[ \nu'(z)-\mu'(z)\right]
\frac{4lj}{j^2-l^2} - \nonumber \\ && (1-\delta_{jl}) \nu'(z)
\frac{4lj}{j^2-l^2} \left[ 1 - (-1)^{l+j} \right] + O(\eps).
\label{eq:PE3}
\end{eqnarray}

Let $\ell^{2}_{1}(\mathbb{Z};L^{2}(\mathbb{R}))$ be the space of
square summable sequences of $L^{2}(\mathbb{R})$ functions with linear
weights, equipped with the norm
\begin{equation*}
\|\textit{\textbf{v}}\|_{\ell_1^2}:=\Big[\sum_{j\in\mathbb{Z}}(j\;\|v_{j}
\|_{L^{2}(\mathbb{R})})^{2}\Big]^{1/2}.
\end{equation*}   
We prove that
$\Psi:\ell^{2}_{1}(\mathbb{Z};L^{2}(\mathbb{R}))\rightarrow
\ell^{2}_{1}(\mathbb{Z};L^{2}(\mathbb{R}))$ is bounded.  The proof
consists of three steps: \\

\textbf{Step 1}: Let $T$ be an
auxiliary operator acting on sequences
$\textit{\textbf{v}}=\{v_{l}\}_{l\in\mathbb{Z}}$, defined
component wise by
\begin{equation*}
[T \textit{\textbf{v}}]_{j}=\sum_{l \neq \pm j} \frac{j\;l}{j^{2}-l^{2}}\;v_{l} = 
\sum_{l \neq \pm j} \left( \frac{l/2}{j+l}+\frac{l/2}{j-l} \right) \;v_{l} = 
\frac{1}{2}\left((-l\;v_{-l})\ast
\frac{1}{l}+(l\;v_{l})\ast
\frac{1}{l}\right)_{j}+\frac{1}{4}(v_{-j}-v_{j}).
\end{equation*}
This operator is essentially the sum of two discrete Hilbert
transforms, satisfying the sharp estimates \cite{Gr}
\begin{equation*}
\|\textit{\textbf{v}}\ast \frac{1}{l}\|_{\ell^2}\leq \pi
\|\textit{\textbf{v}}\|_{\ell^2}.
\end{equation*}
Therefore, the operator $T$ is bounded as  
\begin{equation}\label{op1}
\|T\textit{\textbf{v}}\|_{\ell^2}\leq (1/2+\pi)\;\sum_{j \in
\mathbb{Z}} \|v_j\|_{\ell_1^2}.
\end{equation}
\\ 

\textbf{Step 2}: Let
$\textit{\textbf{v}}(z)=\{v_{l}(z)\}_{l\in\mathbb{Z}}$ be a sequence
of functions in $\mathbb{R}$ and define the operator
\begin{equation}
Q:\ell^{2}_{1}(\mathbb{Z};L^{2}(\mathbb{R})) \to
\ell^{2}_{1}(\mathbb{Z};L^{2}(\mathbb{R})), \qquad
    [Q\textit{\textbf{v}}]_{j}(z)=[T\textit{\textbf{v}}]_{j}\ast
    e^{-\beta_{j}|s|}(z)\;1_{\{j>N\}},
\label{op01}
\end{equation}
where 
\begin{equation}
\beta_{j}=\sqrt{\left(\frac{\pi
    j}{X}\right)^2-\left(\frac{\omega}{c_0}\right)^2} \geq
\frac{j\;\pi}{X}\;\sqrt{1-\left(\frac{\omega X/(\pi
    c_0)}{N+1}\right)^{2}}=: j\;C(\omega), \quad \mbox{for} ~ j > N.
\label{op02}
\end{equation}
Using Young's inequality
\begin{eqnarray}
\|
  [Q\textit{\textbf{v}}]_{j}\|_{L^2(\mathbb{R})}=\|[T\textit{\textbf{v}}]_{j}\ast
  e^{-\beta_{j}|s|}\|_{L^2(\mathbb{R})} \leq
  \|[T\textit{\textbf{v}}]_{j}\|_{L^{2}(\mathbb{R})}\|e^{-\beta_{j}|s|}\|_{L^{1}(\mathbb{R})}=
  \frac{2}{\beta_{j}}\;\| [T\textit{\textbf{v}}]_{j}
  \|_{L^{2}(\mathbb{R})},\label{op03}
\end{eqnarray}
we obtain from \eqref{op1}-\eqref{op03} that $\|Q\| \le (1 + 2 \pi)/C(\om)$, because 
\begin{eqnarray}
\sum_{j\in\mathbb{Z}} \left(j\;\|
    [Q\textit{\textbf{v}}]_{j}\|_{L^2(\mathbb{R})}\right)^{2}&\leq&
    \frac{4}{C(\omega)^{2}}\;\sum_{j\in\mathbb{Z}}\|
         [T\textit{\textbf{v}}]_{j}
         \|^{2}_{L^{2}(\mathbb{R})}=\frac{4}{C(\omega)^{2}}\;\int_{\mathbb{R}}\;
         \sum_{j\in\mathbb{Z}} |[T\textit{\textbf{v}}]_{j}(z) |^{2}
         dz\nonumber \\ &\leq& \frac{4}{C(\omega)^{2}} (1/2+\pi)^{2}
         \int_{\mathbb{R}}\; \sum_{j\in\mathbb{Z}} | j\;v_{j}(z)
         |^{2}dz =\frac{4(1/2+\pi)^{2}}{C(\omega)^{2}}\sum_{j\in\mathbb{R}}
         \left( j \|v_{j}\|_{L^{2}(\mathbb{R})}\right)^2. \qquad 
\label{est1}
\end{eqnarray}

This estimate applies to the operator $\Psi_2$.  Indeed, let us
express $\Psi_2$ in terms of the operator $Q$ using \eqref{eq:Psi2}
and \eqref{eq:PE3},
\begin{equation}\label{defT}
[\Psi_{2} \textit{\textbf{v}}]_{j}(z)=
\frac{1}{2}((\nu'-\mu')v_{j})\ast
e^{-\beta_{j}|s|}(z)1_{\{j>N\}} - 2[Q \mu'
  \;v_{l}]_{j}(z) +2(-1)^{j}[Q \nu'(-1)^{l}\;v_{l}]_{j}(z).
\end{equation}
That the sum in $\Psi_2$ is for $l> N$ is easily fixed by using the
truncation $v_{l}=\hat{v}_{l}\;1_{\{l>N\}}$.  Thus, using estimate
\eqref{est1} for the last two terms, we obtain
\begin{equation*}
\| \Psi_{2} \hat{\textit{\textbf{v}}} \|_{\ell^{2}_{1}}\leq
\frac{5+8\pi}{C(\omega)}\left(\|\mu\|_{W^{1,\infty}(\mathbb{R})}+
\|\nu\|_{W^{1,\infty}(\mathbb{R})}\right)\|\hat{\textit{\textbf{v}}}\|_{\ell^{2}_{1}}.
\end{equation*}
\\

\textbf{Step 3}: It remains to show that the operator $\Psi_1$
is bounded. We see from (\ref{eq:Psi1}), \eqref{eq:PE2} and
\eqref{eq:PE3} that for any $j>N$
\begin{equation*}
[\Psi_{1}\hat{\textit{\textbf{v}}}]_{j}(z)=\frac{\pi^2 j^2}{\beta_{j}
  X^{2}}((\nu-\mu)\hat{v}_{j})\ast
e^{-\beta_{j}|s|}(z)1_{\{j>N\}}-\frac{1}{\beta_{j}}[\tilde{\Psi}_{2}
  \hat{\textit{\textbf{v}}}]_{j}(z),
\end{equation*}
where $\tilde{\Psi}_{2}$ is just like the operator $\Psi_{2}$, with the
driving process $(\nu', \mu')$ replaced by its derivative
$(\nu'',\mu'')$.  Using again Young's inequality, we have 
\begin{align*}
\| [\Psi_{1}\hat{\textit{\textbf{v}}}]_{j}\|_{L^{2}(\mathbb{R})} &\leq
2\left(\frac{\pi}{X
  C(\omega)}\right)^{2}\|(\nu-\mu)\hat{v}_{j}\|_{L^{2}(\mathbb{R})} +\frac{1}{j
  C(\omega) }\| [\tilde{\Psi}_{2}
  \hat{\textit{\textbf{v}}}]_{j}\|_{L^{2}(\mathbb{R})}.
\end{align*}
Now multiply by $j$ and use the triangle inequality to obtain that
$\Psi_1$ is bounded,
\begin{align*}
\|\Psi_{1}\hat{\textit{\textbf{v}}}\|_{\ell^{2}_{1}}&\leq
\left[ \frac{2 \pi^2}{C^2(\om) X^2} \left( \|\nu\|_{L^\infty} +
    \|\mu\|_{L^\infty}\right) + \frac{\left(5 +
  8 \pi\right)}{C^2(\om)}\left( \|\nu\|_{W^{2,\infty}} +
  \|\mu\|_{W^{2,\infty}}\right) \right]
\|\hat{\textit{\textbf{v}}}\|_{\ell^{2}_{1}}.
\end{align*}

\section{Independence of the change of coordinates}
\label{ap:coordc}
We begin the proof of Theorem \ref{thm.2} with the observation that 
\begin{equation*}
\hat{w}(\omega,\xi,z)=\hat{u}\left(
\omega,\ell^{\eps, -1}(z,F^{\eps}(z,\xi)),z\right),
\end{equation*}
where $\ell^{\eps, -1}$ is the inverse of $\ell^{\eps}$,
meaning that $\hat w$ and $\hat u$ are related by composition of the
change of coordinate mappings.  Clearly, the composition inherits the
uniform convergence property
\begin{equation}\label{et1}
\sup_{z\geq0}\sup_{\xi\in[0,X]}|\ell^{\eps, -1}(z,F^{\eps}(z,\xi))-\xi|=
O(\eps).
\end{equation}

For the sake of simplicity we neglect the evanescent modes in the
proof, but they can be added using the techniques described in section
\ref{sect:elim_evanesc}. Using the propagating mode representation of
$\hat{u}(\omega,\xi,z)$,
\begin{align}\label{e1}
\hat{w}(\omega,\xi,z)
&=\sum^{N}_{l=1}\phi_{l}(\omega,\xi)
\hat{u}_{l}(\omega,z)+\sum^{N}_{l=1}\tilde{\phi}_{l}(\omega,\xi,z)
\hat{u}_{l}(\omega,z),
\end{align}
where we let
\begin{align*}
\tilde{\phi}_{l}(\omega,\xi,z) &=\phi_{l}
\left(\omega,\ell^{\eps, -1}(z,F^{\eps}(z,\xi))\right)-
\phi_{l}(\omega,\xi)\\ &=\int^{1}_{0}\left(\ell^{\eps, -1}
(z,F^{\eps}(z,\xi))-\xi\right)\partial_{\xi}\phi_{l}\left(\omega,s\;
\ell^{\eps, -1}(z,F^{\eps}(z,\xi)) +(1-s)\;\xi\right)\;ds.
\end{align*}
But we can also carry out the mode decomposition directly on $\hat w$
and obtain
\begin{equation}\label{e2}
\hat{w}(\omega,\xi,z)=\sum^{N}_{l=1}\phi_{l}(\omega,\xi)\hat{w}_{l}(\omega,z),
\end{equation}
because the number of propagating modes $N$ and the eigenfunctions
$\phi_j$ in the ideal waveguide are independent of the change of
coordinates.  Here $\hat{w}_{l}(\omega,z)$ are the amplitudes of the
propagating modes of $\hat{w}$.  Equating identities \eqref{e1} and
\eqref{e2}, multiplying by $\phi_{j}(\omega,\xi)$ and integrating in
$[0,X]$ we conclude that
\begin{equation}\label{e3}
\hat{w}_{j}(\omega,z)=\hat{u}_{j}(\omega,z)+\sum^{N}_{l=1}
\tilde{c}_{lj}(\omega,z)\hat{u}_{l}(\omega,z),
\end{equation}
where we introduced the random processes,
$$
\tilde{c}_{lj}(\omega,z)= 
\int^{X}_{0}\phi_{j}(\omega,\xi)\int^{1}_{0}
\partial_{\xi}\phi_{l}\left(\omega,s\;\ell^{\eps, -1}(z,F^{\eps}(z,\xi))+(1-s)
\;\xi\right)\left(\ell^{\eps, -1}(z,F^{\eps}(z,\xi))-\xi\right) ds d\xi.
$$
In addition, differentiating equation $\eqref{e3}$ in $z$,  we have
\begin{equation}\label{e4}
\partial_{z}\hat{w}_{j}(\omega,z)=\partial_{z}\hat{u}_{j}(\omega,z)+
\sum^{N}_{l=1}\partial_{z}\tilde{c}_{lj}(\omega,z)\hat{u}_{l}(\omega,z)+
\tilde{c}_{lj}(\omega,z)\partial_{z}\hat{u}_{l}(\omega,z).
\end{equation}

Now, let us recall from the definition of the forward and backward
propagating modes that
\begin{equation*}
i\beta_{j}\hat{u}_{j}(\omega,z)+\partial_{z}\hat{u}_{j}(\omega,z)=
2i\sqrt{\beta_{j}}\;\hat{a}_{j}(\omega,z)e^{i\beta_{j}z}.
\end{equation*}
We conclude from \eqref{e3} and \eqref{e4} that
\begin{multline}\label{e5}
\hat{a}^{w}_{j}(\omega,z)=\hat{a}_{j}(\omega,z) +
\frac{1}{2}\sum^{N}_{l=1}\tilde{c}_{lj}(\omega,z)\left(\frac{\beta_{j}+
  \beta_{l}}{\sqrt{\beta_j
    \beta_{j}}}\;\hat{a}_{l}(\omega,z)e^{-i(\beta_{j}-\beta_{l})z}+
\frac{\beta_{j}-\beta_{l}}{\sqrt{\beta_j\beta_{j}}}\;
\hat{b}_{l}(\omega,z)e^{-i(\beta_{j}+\beta_{l})z}\right)\\ +
\frac{i}{2}\sum^{N}_{l=1}\frac{\partial_{z}\tilde{c}_{lj}(\omega,z)}{\sqrt{\beta_{j}
    \beta_{l}} }
\left(\hat{a}_{l}(\omega,z)e^{-i(\beta_{j}-\beta_{l})z}+
\hat{b}_{l}(\omega,z)e^{-i(\beta_{j}+\beta_{l})z}\right)\, ,
\end{multline}
where $\{\hat{a}^{w}_{j}(\omega,z)\}_{j=1, \ldots, N}$ are the
amplitudes of the forward propagating modes of $\hat{w}(\omega,\xi,z)$.
A similar equation holds for the backward propagating mode amplitudes
$\{\hat{b}^{w}_{j}(\omega,z)\}_{j=1, \ldots, N}$.  

The processes $\tilde{c}_{lj}(\omega,z)$ can be bounded as \eqref{as2}
\begin{eqnarray}
\max_{1\leq j,l\leq N}\{\sup_{z\geq0}|\tilde{c}_{lj}(\omega,z)|\}
\leq X \max_{1\leq j,l\leq N}\{ \sup_{\xi\in[0,X]}
|\phi_{j}(\omega,\xi) |\sup_{\xi\in[0,X]}
|\partial_{\xi}\phi_{l}(\omega,\xi)| \} \; \times \nonumber\\ 
\sup_{z\geq0}\sup_{\xi\in[0,X]}
|\ell^{\eps, -1}(z,F^{\eps}(z,\xi))-\xi| = O(\eps).\label{ct1}
\end{eqnarray}
For the processes $\partial_{z}\tilde{c}_{lj}(\omega,z)$ we find a
similar estimate.  Indeed, note that
\begin{eqnarray*}
&&\partial_{z}\left[\partial_{\xi}\phi_{l}\left(\omega,s\; \ell^{\eps,
      -1}(z,F^{\eps}(z,\xi))+(1-s)\;\xi\right) \left(\ell^{\eps,
      -1}(z,F^{\eps}(z,\xi))-\xi\right)\right] =
  \\&& \hspace{0.2in} -\lambda_{l}\;\phi_{l}(\omega,s\;\ell^{\eps, -1}
  (z,F^{\eps}(z,\xi))+(1-s)\;\xi) \;s\;\partial_{z}[\ell^{\eps,
      -1}(z,F^{\eps}(z,\xi))]\; (\ell^{\eps,
    -1}(z,F^{\eps}(z,\xi))-\xi) + \\ && \hspace{1.9in}
  \partial_{\xi}\phi_{l}(\omega,s\;\ell^{\eps,
    -1}(z,F^{\eps}(z,\xi))+ (1-s)\;\xi)\;\partial_{z}[\ell^{\eps,
      -1}(z,F^{\eps}(z,\xi))].
\end{eqnarray*}
A direct calculation shows that
\begin{align*}
\partial_{z} &
\left[\ell^{\eps, -1}(z,F^{\eps}(z,\xi))\right]=\partial_{z}\left[
  \frac{ X( F^{\eps}(z,\xi)-\eps\mu(z) ) }{
    X(1+\eps\nu(z))-\eps\mu(z) }
  \right]\\ &=X\frac{(\partial_{z}F^{\eps}(z,\xi)-
  \eps\mu'(z))(X(1+\eps\nu(z))-\eps\mu(x))-(
  F^{\eps}(z,\xi)-\eps\mu(z) )\;\eps\;( \nu'(z)-\mu'(z)
  )}{(X(1+\eps\nu(z))-\eps\mu(x))^{2}}.
\end{align*}
Hence, using condition \eqref{as2} for
$\partial_{z}F^{\eps}(z,\xi)$
\begin{equation*}
\sup_{z\geq0}\sup_{\xi\in[0,X]} \left|\partial_{z}\left
    [\ell^{\eps,  -1}(z,F^{\eps}(z,\xi))\right]\right|\leq
    C(\|v\|_{W^{1,\infty}},\|\mu\|_{W^{1,\infty}})\;\eps.
\end{equation*}
Therefore,
\begin{eqnarray}
\max_{1\leq j,l\leq
  N}\{\sup_{z\geq0}|\partial_{z}\tilde{c}_{lj}(\omega,z)|\}\leq
X\max_{1\leq j,l\leq N}\{ \lambda_{l} \sup_{\xi\in[0,X]}
|\phi_{j}(\omega,\xi) |\sup_{\xi\in[0,X]} |\phi_{l}(\omega,\xi)|
\}\;O(\eps^{2}) + \nonumber\\  X\max_{1\leq j,l\leq N}\{
\sup_{\xi\in[0,X]} |\phi_{j}(\omega,\xi) |\sup_{\xi\in[0,X]}
|\partial_{\xi}\phi_{l}(\omega,\xi)| \}\;O(\eps).\label{ct2}
\end{eqnarray}

Let $\hat{\itbf a}^w(\omega,z)$ and $\hat{\itbf b}^w(\omega,z)$ be
the vectors containing the forward and backward propagating mode amplitudes and define
the joint process of propagating mode amplitudes $\bX_\om^{
  w}(z)=(\hat{\itbf a}^w(\omega,z),\hat{\itbf b}^w(\omega,z))^{T}$.
Let us the long range scaled process be $\bX_\om^{\eps, w}(z) =
\bX_\om^{ w}(z/\eps^2).$ Equation \eqref{e5} implies that
\begin{equation}\label{e6}
\bX_\om^{\eps,w}(z)=
\bX_\om^{\eps}(z)+
\bold{M}_{\eps}\left(\omega,\bold{C}\Big(\omega,\frac{z}{\eps^2}\Big),
\partial_{z}\bold{C}\Big(\omega,\frac{z}{\eps^{2}}\Big),\frac{z}{\eps^{2}}\right)
\bX_\om^\eps(z),
\end{equation}
where $\bold{C}(\omega,z):=(\tilde{c}_{lj}(\omega,z))_{j,l=1,\ldots,N}$ and
$\partial_{z}\bold{C}(\omega,z):=(\partial_{z}\tilde{c}_{lj}(\omega,z))_{j,l=1,\ldots,N}$.
 The subscript $\eps$ in the matrix
$\bold{M}_{\eps}(\cdot)$ denotes the fact that this matrix depends
explicitly on $\eps$ and, due to estimates \eqref{ct1} and
\eqref{ct2}, we have 
\begin{equation}\label{e7}
\sup_{z\geq0}\|\bold{M}_{\eps}(\omega,\bold{C}(\omega,z),\partial_{z}
\bold{C}(\omega,z),z)\|_{\infty}=O(\eps).
\end{equation} 

Let us prove then, that the processes $\bX_\om^{\eps, w}(z)$ and
$\bX_\om^\eps(z)$ converge in distribution to the same diffusion
limit.  Denote by $Q(\bX_{0},L)$ the $2N$-dimensional cube with
center $\bX_0$ and side $L$.  The probability that
$\bX_\om^{\eps, w}(z)$ is in this cube can be calculated using
\eqref{e6}, 
\begin{align}
\label{dl1}
\mathbb{P}[\bX^{\eps,w}_\om(z)\in
  Q(\bX_0,L)]&=\int_{\{ \bx \in
  Q(\bX_{0},L)\}}\;d\mathbb{P}^{w}\left(\bx,
\frac{z}{\eps^{2}}\right)\nonumber\\ &=\int_{\{\bx \in
  ({\bf I}+\bold{M}_{\eps}(\bold{C},\partial_{z}\bold{C},z))^{-1}
  Q( \bx_{0},L) \}}\;d\mathbb{P}
\left(\bx,\bold{C},\partial_{z}\bold{C},\frac{z}{\eps^{2}}\right).
\end{align}
Here $\mathbb{P}^{w}(\bx,z)$ is the probability distribution of
the process $\bX^{w}_\om(z)$ and
$\mathbb{P}\left(\bx,\bold{C},\partial_{z}\bold{C},z\right)$ is
the joint probability distribution of the processes
$(\bX_\om(z),\bold{C}(\omega,z),\partial_{z}\bold{C}(\omega,z))$.
We can take the inverse of
${\bf I}+\bold{M}_{\eps}(\bold{C},\partial_{z}\bold{C},z)$ by
\eqref{e7}.  The same estimate \eqref{e7} also implies that for every
$\delta>0$ there exists $\eps_{0}$ such that for
$\eps\leq\eps_0$,
\begin{equation}\label{dl2}
\{\bx\in Q(\bx_{0},(1-\delta)L)\}\subseteq\{\bx\in
({\bf I}+\bold{M}_{\eps}(\bold{C},\partial_{z}\bold{C},z))^{-1}
Q(\bx_{0},L)\}\subseteq\{\bx\in
Q(\bx_{0},(1+\delta)L)\}.
\end{equation}
Denote the diffusion limits by 
\begin{align*}
\tilde{\bX}_\om (z)=\lim_{\eps\rightarrow0}
\bX^\eps_\omega(z), \qquad  \tilde{\bX}^{w}_\omega(z)=\lim_{\eps\rightarrow0}
\bX^{\eps, w}_\omega(z).
\end{align*}
We conclude from \eqref{dl1} and \eqref{dl2} that for any $\delta>0$,
\begin{align*}
\mathbb{P}[\tilde{\bX}_\om(z)\in Q(\bX_0,(1-\delta)L)]\leq
\mathbb{P}[\tilde{\bX}_\om^{w}(z)\in Q(\bX_0,L)] \leq
\mathbb{P}[\tilde{\bX}_\om(z)\in Q(\bX_0,(1+\delta)L)].
\end{align*}
Sending $\delta\rightarrow 0$, we have that for any arbitrary cube $
Q(\bx_0,L)$
\begin{equation*}
\mathbb{P}[\tilde{\bX}_\omega(z)\in
  Q(\bX_0,L)]=\mathbb{P}[\tilde{\bX}^{w}_\omega(z)\in
  Q(\bX_0,L)].
\end{equation*}
This proves that the limit processes have the same distribution and
therefore, the same generator.

\section{Proof of Proposition  \ref{prop.estim}}
\label{ap:estim}
Recall the expression \eqref{eq:1} of the wavenumbers.  The first term
in \eqref{eq:EE2} follows from \eqref{defgamma10}:
\begin{equation}
\Gamma_{jj}^{(0)} = \left( \frac{\pi}{X} \right)^2 \left[ \hat
  \cR_\nu(0) + \hat \cR_\mu(0) \right] \frac{j^4}{(N+\alpha)^2 - j^2}
\approx \frac{(2 \pi)^{3/2}}{X} \frac{k \ell}{N}
\frac{j^4}{(N+\alpha)^2 - j^2} \,.
\label{eq:EE6}
\end{equation}
It increases
monotonically with $j$, with minimum value
\begin{eqnarray}
\Gamma_{11}^{(0)} \approx \frac{(2 \pi)^{3/2}}{X} \frac{k \ell }{N^3}
\ll 1\,,
\label{eq:EE6_1}
\end{eqnarray}
and maximum value
\begin{eqnarray}
\Gamma_{NN}^{(0)} \approx  \frac{(2 \pi)^{3/2}}{2 \alpha X} k \ell N^2 \gg 1 \, .
\label{eq:EE_2}
\end{eqnarray}

The second term in \eqref{eq:EE2}, which is in (\ref{eq:EE2J}),  follows from \eqref{defgamma1b}, \eqref{eq:PSGaussian} and 
\eqref{eq:EE5},
\begin{eqnarray}
- \Gamma_{jj}^{(c)}(\om) &\approx& \frac{(2 \pi)^{3/2} j^2 }{X
  \sqrt{(N+\alpha)^2 - j^2}}\hspace{-0.05in} \sum_{{\scriptsize
  \begin{array}{c}l = 1 \\ l \ne j \end{array}}}^N
  \hspace{-0.05in}\frac{l^2 k \ell}{N\sqrt{ (N+\alpha)^2 - l^2}}
 e^{ -\frac{(k \ell)^2}{2} \left( \sqrt{1 - j^2/(N+\alpha)^2} -
 \sqrt{1-l^2/(N+\alpha)^2}\right)^2}  \, . \quad
\label{eq:EE8}
\end{eqnarray}

If $0< j/N <1$,
then we can estimate (\ref{eq:EE8})
by using the fact that the main contribution to the sum in~$l$ comes from the terms with indices $l$ close to $j$,
provided that  $k \ell$ is larger than $N^{1/2}$ and smaller than $N$.
We find after the change of index $l=j+q$:
\begin{eqnarray*}
- \Gamma_{jj}^{(c)}(\om) &\approx& \frac{(2 \pi)^{3/2} j^4 k \ell}{X
 ((N+\alpha)^2 - j^2) N} \sum_{q\neq 0 }
 e^{ -\frac{(k \ell)^2}{2} \frac{j^2}{(N+\alpha)^2-j^2} \frac{q^2}{(N+\alpha)^2}} 
\end{eqnarray*}
Interpreting this sum as the Riemann sum of a continuous integral, we get
\begin{eqnarray}
- \Gamma_{jj}^{(c)}(\om)  &\approx& \frac{(2 \pi)^{3/2} j^4 k \ell}{X
 ((N+\alpha)^2 - j^2)}   \int_{-\infty}^\infty
 e^{ -\frac{(k \ell)^2}{2} \frac{j^2}{(N+\alpha)^2-j^2} s^2 } ds = 
  \frac{(2 \pi)^2 j^3}{X \sqrt{(N+\alpha)^2 - j^2}} .
  \label{eq:EE8b}
\end{eqnarray}
By comparing with (\ref{eq:EE6}) 
we find that the coefficient $- \Gamma_{jj}^{(c)}(\om)$ is larger than $\Gamma_{jj}^{(0)}$ when $k \ell$ satisfies 
$\sqrt{N}\ll k \ell \ll N$.\\
To be complete, note that:\\
- If $k \ell \sim N$, then  $- \Gamma_{jj}^{(c)}(\om)$ is larger than $\Gamma_{jj}^{(0)}$ if and only if 
$j/N < (1+(k\ell /N)^2 ) ^{-1/2}$.\\
- If $k \ell$ is larger than $N$, then the main contribution to the sum in $l$ comes only from
one or two terms with indices $l=j\pm 1$, and it becomes exponentially small in $ (k \ell)^2/ N^2$.
In these conditions $- \Gamma_{jj}^{(c)}(\om)$ becomes smaller than $\Gamma_{jj}^{(0)}$.\\

%
 
For $j \sim 1$ we can estimate \eqref{eq:EE8} again by
interpreting the sum over $l$ as a Riemann sum approximation of an
integral that we can estimate using the Laplace perturbation method.
Explicitly, for $j=1$ we have
\begin{eqnarray}
- \Gamma_{11}^{(c)}(\om) &\approx& \frac{(2 \pi)^{3/2}}{X} \frac{1}{N}
  \sum_{l = 2}^N
  \hspace{-0.05in}\frac{(l/N)^2 k \ell}{\sqrt{ (1+\alpha/N)^2 -
 (l/N)^2}} e^{ -\frac{(k \ell)^2}{2} \left( 1-
 \sqrt{1-(l/N)^2}\right)^2} \, \nonumber \\ &\approx & \frac{(2
 \pi)^{3/2}k \ell }{X} \int_{0}^1 ds \frac{s^2}{\sqrt{1 -
 s^2}} e^{-\frac{(k \ell)^2}{2} \left(1 - \sqrt{1-s^2}\right)^2}.
\label{eq:EE10}
\end{eqnarray}
We approximate the integral with Watson's lemma \cite[Section
  6.4]{bender}, after changing variables $ \zeta = (1-\sqrt{1-s^2})^2$
and obtaining that
\[
\int_{0}^1 ds \frac{s^2}{\sqrt{1 -
 s^2}} e^{-\frac{(k \ell)^2}{2} \left(1 - \sqrt{1-s^2}\right)^2} 
\approx \int_0^1 d \zeta \varphi(\zeta) e^{-\frac{(k \ell)^2}{2} \zeta}, \qquad 
\varphi(\zeta) = \frac{\zeta^{-1/4}}{\sqrt{2}} + O(\zeta^{1/4})\, .
\]
Watson's lemma gives 
\[
\int_{0}^1 ds \frac{s^2}{\sqrt{1 - s^2}} e^{-\frac{(k \ell)^2}{2}
  \left(1 - \sqrt{1-s^2}\right)^2} \approx \frac{\Gamma(3/4)
  2^{1/4}}{(k \ell)^{3/2}} \, , 
\]
and therefore by \eqref{eq:EE10} and \eqref{eq:kllarge},
\begin{eqnarray}
- \Gamma_{11}^{(c)}(\om) &\approx& \frac{(2 \pi)^{3/2} \Gamma(3/4)
  2^{1/4}}{X (k \ell)^{1/2}} 
\, .
\end{eqnarray}
By comparing with (\ref{eq:EE6_1}) 
we find that the coefficient $- \Gamma_{11}^{(c)}(\om)$ is larger than $\Gamma_{11}^{(0)}$.\\

For $j\sim N$  only the terms with $l \sim N$ contribute to the sum in
\eqref{eq:EE8}.
If $k \ell \sim \sqrt{N}$, then we find that 
$$
- \Gamma_{NN}^{(c)} (\omega) \approx \frac{(2\pi)^{3/2} N^2  k \ell}{2 \sqrt{\alpha} X} \sum_{q=1}^\infty \frac{1}{\sqrt{\alpha+q}} 
e^{ -\frac{(k\ell)^2}{2 N} (\sqrt{q+\alpha}- \sqrt{\alpha} )^2 }
\sim  \frac{(2\pi)^{3/2} N^3} {2 C(\alpha) k\ell X}  ,
$$
up to a constant $C(\alpha)$ that depends only on $\alpha$.
By comparing with (\ref{eq:EE_2}) we can see that  it is of the same order  as $\Gamma_{NN}^{(0)}$.
If $k \ell \gg \sqrt{N}$, then we find that 
$$
- \Gamma_{NN}^{(c)} (\omega) \approx \frac{(2\pi)^{3/2} N^2  k \ell}{2 \sqrt{\alpha(1+\alpha)} X}  
e^{ -\frac{(k\ell)^2}{2 N} (\sqrt{1+\alpha}- \sqrt{\alpha} )^2 } ,
 $$
which is very small because the exponential term is exponentially small in $(k\ell)^2/ N$.
In these conditions $- \Gamma_{NN}^{(c)} (\omega) $ is smaller than $\Gamma_{NN}^{(0)}$.

 \end{document}